\documentclass[11pt]{amsart}
\usepackage{amsmath,amsthm,amsfonts,amssymb}
\usepackage{verbatim}
\usepackage{latexsym}
\usepackage{nicefrac}
\usepackage{dsfont}
\usepackage{color}
\usepackage[normalem]{ulem}
\usepackage{enumerate}
\usepackage{mdwlist}
\usepackage[british]{babel}
\usepackage[utf8]{inputenc}
\usepackage{soul}
\usepackage{graphics} 
\usepackage{tikz}
\usepackage{caption} 
\usepackage{xurl}
\usepackage[hidelinks]{hyperref} 
\usepackage{amssymb}
\usepackage{amsmath}
\usepackage{amsthm}

\usepackage{dsfont}
\usepackage{nicefrac}
\usepackage{caption}
\usepackage[
backend=biber,
doi=false,
url=false,
eprint=false,
isbn=false
]{biblatex}
\usepackage{booktabs} 

\usepackage{pgf} 

\newcommand{\R}{{\mathbb{R}}}
\newcommand{\C}{{\mathbb{C}}}
\newcommand{\N}{\mathbb{N}}
\newcommand{\Z}{\mathbb{Z}}
\newcommand{\h}{\mathcal{H}}
\newcommand{\s}{\mathcal{S}}

\newcommand{\dx}{\:\mathrm{d}\,x\:}
\newcommand{\dxy}{\:\mathrm{d\lambda}\,(x,y)\:}

\newcommand{\dz}{\:\mathrm{d}\,z\:}
\newcommand{\dxi}{\:\mathrm{d}\,\xi\:}
\newcommand{\dnu}{\:\mathrm{d}\,\nu\:}

\newcommand{\abs}[1]{\ensuremath{\left|\text{}#1\text{}\right|}}
\newcommand{\norm}[1]{\ensuremath{\left\lVert\text{}#1\text{}\right\rVert}}
\newcommand{\ip}[1]{\ensuremath{\left<\text{} \, #1 \, \text{}\right>}}
\newcommand{\set}[1]{\ensuremath{\left\{\text{}#1\text{}\right\}}}

\DeclareMathOperator{\re}{Re}
\DeclareMathOperator{\sgn}{sgn}
\DeclareMathOperator{\id}{id}
\DeclareMathOperator{\GL}{GL}
\DeclareMathOperator{\Wav}{Wav}
\DeclareMathOperator{\supp}{supp}

\DeclareMathOperator{\vol}{vol}
\DeclareMathOperator*{\argmin}{argmin}

\newtheorem{theorem}{Theorem}[section]
\newtheorem{lemma}[theorem]{Lemma}
\newtheorem{proposition}[theorem]{Proposition}
\newtheorem{definition}[theorem]{Definition}
\newtheorem{corollary}[theorem]{Corollary}

\newtheorem{example}[theorem]{Example}

\newtheorem{remark}[theorem]{Remark}
\newtheorem{assumption}[theorem]{Assumption}

\newenvironment{proof sketch} {\textit{Proof sketch.}}{}

\newenvironment{myenum}
{ \begin{enumerate}
    \setlength{\itemsep}{3pt}
    \setlength{\parskip}{0pt}
    \setlength{\parsep}{0pt}     }
{ \end{enumerate}                  }

\begin{document}

\title[Energy Propagation in Scattering Convolution Networks]
{Energy Propagation in Scattering Convolution Networks Can Be Arbitrarily Slow}
\author{Hartmut F\"uhr, Max Getter}
\email{fuehr@mathga.rwth-aachen.de, getter@mathga.rwth-aachen.de}
\address{Chair for Geometry and Analysis, RWTH Aachen University, D-52062 Aachen, Germany}

\maketitle

\vspace*{-2em}

\begin{abstract}
  We analyze energy decay for deep convolutional neural networks employed as feature extractors, including Mallat's wavelet scattering transform. For time-frequency scattering transforms based on Gabor filters, previous work has established that energy decay is exponential for arbitrary square-integrable input signals. In contrast, our main results allow proving that this is false for wavelet scattering in arbitrary dimensions. Specifically, we show that the energy decay of wavelet and wavelet-like scattering transforms acting on generic square-integrable signals can be arbitrarily slow. Importantly, this slow decay behavior holds for dense subsets of $L^2(\mathbb{R}^d)$, indicating that rapid energy decay is generally an unstable property of signals. We complement these findings with positive results that allow us to infer fast (up to exponential) energy decay for generalized Sobolev spaces tailored to the frequency localization of the underlying filter bank. Both negative and positive results highlight that energy decay in scattering networks critically depends on the interplay between the respective frequency localizations of both the signal and the filters used. 
  \end{abstract}
  \vspace*{1em}
  
  \noindent \textbf{\small Keywords:}{\small {}  Deep convolutional neural networks; scattering transform; energy propagation; filter banks; wavelets}{\small \par}
  
  \noindent \textbf{\small 2020 Mathematics Subject Classification:}{\small {} 42C15; 68T07; 42C40; 42B35}{\small \par}
  
  \begingroup
  \renewcommand\thefootnote{}\footnotetext{%
  	\textcopyright\ 2025. Author-created version of an article published in
  	\emph{Applied and Computational Harmonic Analysis} (Open Access, CC BY 4.0).
  	DOI: https://doi.org/10.1016/j.acha.2025.101790. Minor change vs. published version: added one bibliographic reference.
  }
  \addtocounter{footnote}{-1}
  \endgroup
  
  \section{Introduction}
  
  \subsection{Motivation and related work}\label{sec: Motivation and related work}
  
  Up to this day, the tremendous success of convolutional neural networks (CNNs) in computer vision tasks \cite{zhao2024review} like image classification is only partially understood. This outstanding performance of CNNs is generally believed to be due to their ability to capture information at multiple scales through the use of convolutions and subsequently aggregating the information observed at these scales by means of further operations such as pooling. Consequently, CNNs are seemingly able to extract semantic content from images while discarding irrelevant information. However, the available rigorous analytical explanations for their success are still incomplete.
  
  In his pioneering work \cite{mallat2012group}, Stéphane Mallat introduced the 
  scattering transform, a nonlinear operator that cascades convolutions followed by the modulus non-linearity. The scattering transform can be viewed as an infinitely deep and infinitely wide CNN with the modulus as non-linearity, predetermined filters, and no pooling \cite{wiatowski2017mathematical}. Due to the use of hand-crafted filters based on a wavelet construction, there is no need to train the network. The windowed scattering transform was extended in \cite{wiatowski2017mathematical} to allow for the use of other (optionally learned) filters that are not wavelet-generated, with potentially different filters for different layers, other non-linearities than the modulus, and pooling operations between layers. The present work studies the scattering transform in a more classical fashion, with, for simplicity of exposition, identical (not necessarily wavelet-generated) filters across network layers, the modulus as non-linearity, and no pooling. 
  
  Despite the absence of a training phase and pooling, Mallat-type scattering networks are still performing comparably well to other state-of-the-art models, even after more than a decade of progress in machine learning since their invention (e.g., \cite{su2023wavelets,zhao20233d}).  
  They are used in various applications such as vision and audio tasks, quantum chemistry, medicine, astronomy, and manifold learning \cite{anden2015joint,chew2024geometric,eickenberg2018solid,gao2019geometric,hirn2017quantum,sinz2020wavelet,tolley2024wavelet,tschannen2016heart}. In many of these, the scattering transform (or a variant thereof) is employed in the preprocessing phase of the input data in that it acts as a feature extractor, which is supposed to improve the overall performance of the machine learning model.
  
  Such scattering networks satisfy many of the nice properties that are generally believed to be essential in a feature extractor \cite{mallat2012group,mallat2016understanding,wiatowski2017mathematical,nicola2023stability}. 
  One of these properties, which is of particular importance both from a practical as well as a theoretical point of view, is the question of fast energy propagation across the network layers: Under mild assumptions on the filters, the energy of any input signal 
  decomposes, for every $N \in \N$, into the aggregated energy that is contained in the first $N$ layers of the scattering network 
  and some energy remainder (see Proposition \ref{prop: Energy decomposition and norm-preservation of S} for details). Thus, the scattering transform preserves the energy of an input signal if and only if the corresponding energy remainder converges to zero as $N\to \infty$. In this sense, all of the original information about the input signal is then still contained in the features generated by the scattering network. In practice, however, only finitely many features of the infinite scattering network can be computed. When restricting the network to finitely many features of the first $N$ layers, the corresponding energy remainder is a lower bound for the loss of information. Hence, fast decay of this quantity is desirable.
  
  Energy conservation in the above sense was first established in \cite[Theorem 2.6]{mallat2012group} for wavelet-generated filter banks under a rather restrictive and technical admissibility assumption. Under less restrictive assumptions on the wavelets, but only in dimension $d=1$, the decay of the energy remainder was shown \cite[Theorem 3.1]{waldspurger2017exponential} to be at least of exponential order $\mathcal{O}(\alpha^{N})$ for Sobolev functions with some \textit{unspecified}, signal-independent $\alpha \in (0,1)$. In recent years, quantitative results in a similar fashion were also derived for other types of filter banks. Energy decay was proven (in arbitrary dimension) to be at least of exponential order $\mathcal{O}(\alpha^{N})$ for all finite-energy input signals, with some unspecified, signal-independent $\alpha \in (0,1)$, if the filters of the underlying scattering network form a so-called uniform covering frame \cite[Proposition 3.1]{czaja2019analysis}. In particular, this class of filter banks includes certain Weyl-Heisenberg (Gabor) frames \cite[Proposition 2.3]{czaja2019analysis}. However, the uniform covering property requires the Fourier supports of the filters to be uniformly bounded, so that a broad range of commonly used filter banks in signal processing (e.g., wavelet-generated filter banks) is not addressed by this result. 
  
  With refined estimates, but based on an idea similar to that of \cite[Theorem 3.1]{waldspurger2017exponential}, energy decay was proven to be at least of exponential order $\mathcal{O}(\alpha^{N})$ for Sobolev functions if the filter bank is generated by a \textit{bandlimited} wavelet in dimension $d=1$, with a concretely specified decay factor $\alpha$ that depends on the bandwidth of the generator \cite[Theorem 2]{wiatowski2017energy}, \cite[Theorem 3.1]{wiatowski2017topology}. If the filters do not necessarily have an underlying structure as in the case of wavelet-generated filters, but instead satisfy certain mild analyticity assumptions in arbitrary dimension $d \in \N$, energy decay was shown to be at least of polynomial order $\mathcal{O}(N^{-m_d})$ for Sobolev functions, where $m_d \in (0,1]$ and $m_d \to 0$ as $d \to \infty$ \cite[Theorem 1]{wiatowski2017energy}. 
  
  So far, an important question has remained unanswered: 
  
  \textit{Does exponential energy decay hold for the wavelet scattering transform, globally on $L^2(\R^d)$?}
  
  \subsection{Contributions}
  
  Our first main result, Theorem \ref{thm: Arbitrarily slow convergence of W_N}, provides a negative answer to the question of global exponential energy decay for the wavelet scattering transform. In fact, we show that energy decay can even become arbitrarily slow over $L^2(\R^d)$ when employing filter banks that have an underlying structure similar to wavelet-generated filter banks (including the latter). It turns out that, for given decay rate, the set of signals that do not obey this rate in terms of scattering energy decay is dense in $L^2(\R^d)$. Informally, energy decay can thus be considered instable over $L^2(\R^d)$ for these types of scattering networks. 
  
  We complement these negative findings with our second main result, Theorem \ref{thm: Convergence rates for W_N(f)}, which is of positive nature. By exploiting the interplay between Fourier decay of the input signal and frequency concentration of the underlying filters of the scattering network, we provide explicit upper bounds on the convergence rates of the energy remainder for large filter-dependent subclasses of $L^2(\R^d)$, thereby generalizing, unifying, and partially improving the findings of the previous works \cite{czaja2019analysis,waldspurger2017exponential,wiatowski2017energy,wiatowski2017topology}.
  
  In view of selecting wavelet-generated filters for the scattering network (as is the case in many applications), these results have two main implications:
  \begin{itemize}
    \item Fast energy propagation in wavelet scattering networks can only be guaranteed if a priori knowledge about the global Fourier decay of the input data is available, with sufficient conditions specified by Corollary \ref{cor: Wavelet decay rates}.
    \item The set of those signals in $L^2(\R^d)$, for which the mixed $(\ell^1,L^2(\R^d))$ scattering norm considered in \cite[Section 2.5]{mallat2012group} is infinite, is dense in $L^2(\R^d)$. Results in the spirit of \cite[Theorem 2.12]{mallat2012group}, which guarantee stability of the windowed scattering transform under the action of small diffeomorphisms for signals with finite mixed $(\ell^1,L^2(\R^d))$ scattering norm, therefore do not apply to such (in this sense) ill-behaved input signals. On the other hand, finiteness of this mixed scattering norm can be guaranteed by means of our positive results, e.g., for all $f \in H^s_{\log}(\R^d)$, $s>1$, which generalizes \cite[Proposition 2.4]{nicola2023stability} to arbitrary dimension $d \in \N$ if the generating wavelet is bandlimited.
  \end{itemize}

  \subsection{Structure of the paper}
  
  We begin our exposition in Section \ref{sec: brief review} by reviewing the windowed scattering transform and clarifying our base assumptions on the underlying filters of the scattering network. 
  
  Our results in Section \ref{sec: Slow Scattering Propagation} concern filter banks with an inherent structure similar to those generated by wavelets. The starting point of our analysis is Lemma \ref{lem: approx additivity}, which states an approximate version of super-additivity of the energy remainder for signals that are nearly separated in frequency domain. We build on this result by constructing an adversarial signal by means of an infinite series of appropriately chosen, nearly pairwise separated (in frequency domain) signals.
  
  In Section \ref{sec: Convergence rates for scattering propagation}, we derive upper bounds on the decay rate of the energy remainder, which are based on the interplay between Fourier decay of the input signal and the frequency concentration of the filters. Our results apply to a broad range of filter banks and their corresponding signal classes. The filters must satisfy a mild analyticity assumption, which can be seen as a refinement of the setting considered in \cite[Assumption 1]{wiatowski2017energy}. 
  
  Finally, we illustrate our results in Section \ref{sec: Applications} by outlining the consequences for a certain class of filters with uniformly bounded bandwidth, and for wavelet-generated filters. In doing so, we recover some existing results from the literature, and we improve and generalize others. These results are of considerable independent interest, which is further enhanced by the contrast to the negative findings from Section \ref{sec: Slow Scattering Propagation}.
  
  \subsection{Notation}\label{sec: Notation}
  
  This paper has a significant amount of notation, which we summarize here categorically. Table \ref{tab: notation} provides an overview thereof.
  
  \begin{table}[h!]
  	\centering
  	\caption{Overview of notation used in the paper.}
  	\label{tab: notation}
  	\resizebox{\textwidth}{!}{
  		\begin{tabular}{@{}lll@{}}
  			\toprule
  			\textbf{Notation} & \textbf{Meaning} & \textbf{Reference} \\ \midrule
  			$d$ & Dimension of the domain of the input signal &  \\
  			$S_{\geq x}$ & 
  			Vectors in $S\subseteq \R^d$ that satisfy the inequality component-wise for fixed $x \in \R^d$ & Sec. \ref{sec: Notation} \\
  			$L^p(\R^d)$ & Lebesgue space of complex-valued functions with exponent $p$ & Sec. \ref{sec: Notation} \\ 
  			$\ell^p(\Lambda;\h)=\ell^p(\h)$ & Space of 
  			$p$-summable maps from a countable set $\Lambda$ into a Hilbert space $\h$  & Eq. \eqref{def: vector-valued sequence space} \\
  			$\ip{\cdot,\cdot}$ & Standard inner product on $\C^d$, $L^2(\R^d)$, or $\ell^2(\h)$ & Sec. \ref{sec: Notation}\\
  			$\norm{\cdot}_p$ & Norm on $L^p(\R^d)$ or $\ell^p(\h)$ & Sec. \ref{sec: Notation}\\
  			$\GL_d(\R)$ & General linear group in dimension $d$ over the field $\R$ & \\
  			$O_d(\R)$ & Orthogonal group in dimension $d$ over the field $\R$ & \\
  			$I_d$ & Identity matrix in dimension $d$ & \\
  			$\sigma_{\text{min}}(A),\sigma_{\text{max}}(A)$ & A smallest (respectively largest) singular value of a matrix $A$ & \\
  			$B_r(x)$ & Open Euclidean ball with center $x$ and radius $r$ & \\
  			$\s_{r,R}$ & Closed spherical shell in $\R^d$ with center $0$, inner radius $r$, and outer radius $R$ & Sec. \ref{sec: Notation} \\
  			$C^\rho$, $C^\rho_{A}$ & Closed cone with tip at the origin, opening angle $\rho$, and orientation $A \in O_d(\R)$ & Eq. \eqref{def: cone} \\
  			$\mathds{1}_M$ & Indicator function on the set $M$ & \\
  			$\vol(M)$ & $d$-dimensional Lebesgue-measure of the set $M$ & \\
  			$\mathbb{S}^{d-1}$ & $d-1$-dimensional sphere in $d$-dimensional Euclidean space & Sec. \ref{sec: Notation} \\
  			$W_N \in \mathcal{O}(E_N)$ & Bachmann-Landau asymptotic big O notation for sequences $E$, $W$ & Sec. \ref{sec: Notation}\\
  			$\supp(f)$ & Support of a function $f$ & Sec. \ref{sec: Notation}\\
  			$D_A^pf$ & $L^p$-normalized dilation of a function $f$ by $A\in \GL_d(\R)$, $p \in \{1,2\}$ & Sec. \ref{sec: Notation} \\
  			$f^*$ & Reflection and conjugation of a function $f$ & Sec. \ref{sec: Notation}\\
  			$\mathcal{F}f=\widehat{f}$ & Fourier transform of a function $f$ with normalization $2\pi i$ & Eq. \eqref{def: Fourier transform} \\
  			$\mathcal{F}L_\omega^2(\R^d)$ & Sobolev space with weight $\omega$ & Eq. \eqref{def: weighted Sobolev space} \\
  			$H^s(\R^d)$ & Fractional Sobolev space of class regularity $s$ & Sec. \ref{sec: Notation} \\
  			$H^s_{\log}(\R^d)$ & Fractional Log-Sobolev space of class regularity $s$ & Sec. \ref{sec: Notation} \\
  			$\mathfrak{F}=\{\chi\} \cup \Psi$ & Semi-discrete Parseval frame with low-pass filter $\chi$ & Sec. \ref{sec: brief review} \\
  			$d_\psi$ & Diameter of the spectral support of a filter $\psi$ & Eq. \eqref{def: diameter of spectral support of a filter}\\
  			$D_\Psi$ & Supremum of the diameters of the spectral supports of filters in $\Psi$ & Def. \ref{def: uniform frequency concentration}\\
  			$\mathcal{D}_\omega(\Psi;L^2(\R^d))$ & Generalized Sobolev space with weight $\omega$ and filters $\Psi$ & Eq. \eqref{def: generalized Sobolev space} \\
  			$Y_E$ & Signals whose energy remainder decays at least at the rate of $E$ & Cor. \ref{cor: arbitrarily slow decay if delta is positive} \\
  			$U[\psi]$ & Single step scattering propagator with filter $\psi$ & Eq. \eqref{def: single step scattering propagator} \\
  			$U[p]$ & Scattering propagator along filter path $p$ & Eq. \eqref{def: scattering propagator} \\
  			$S[p;\chi]$ & Windowed scattering propagator along path $p$ and with window $\chi$ & Eq. \eqref{def: windowed scattering propagator} \\
  			$U[P]$, $S[P;\chi]$ & Collection of (windowed) scattering propagators along paths in $P$ & Eq. \eqref{def: collection of scattering propagators}, \eqref{def: collection of windowed scattering propagators} \\
  			$\s[\mathfrak{F}]$ & Windowed scattering transform with filters $\mathfrak{F}$ & Eq. \eqref{def: windowed scattering transform} \\
  			$W_N[\Phi](f)$ & Energy remainder associated with depth $N$, filters $\Phi\subseteq \mathfrak{F}$, and signal $f$ & Eq. \eqref{def: energy remainder}\\
  			$W_N(f)$ & Quantity of main interest; $W_N(f):=W_N[\Psi](f)$ & Eq. \eqref{def: energy remainder} \\
  			\bottomrule
  	\end{tabular}}
  \end{table}

  Fix a dimension $d \in \N$. The sets $\N^d\subseteq\N_0^d\subseteq\Z^d\subseteq \R^d \subseteq \C^d$ have their usual meanings. For $S\subseteq \R^d$ and $x \in \R^d$, we define $S_{\geq x}:=\{s \in S ~|~ s_k \geq x_k ~ \text{ for all } k \in \{1,\ldots,d\}\}$. We denote the Euclidean inner product of two vectors $z,w \in \C^d$ by $\ip{z,w}:=\sum_{k=1}^d z_k \overline{w_k}$, with associated norm $\norm{\cdot}_2$, and the respective open ball in $\R^d$ with center $x \in \R^d$ and radius $r>0$ by $B_r(x)\subseteq \R^d$. Further, we write $\s_{r,R}:=\overline{B_R(0)}\setminus B_r(0)$ for the closed spherical shell with center at the origin, inner radius $r>0$, and outer radius $R>0$. By $\mathds{1}_M$, we refer to the indicator function of a set $M \subseteq \R^d$. If $M$ is Lebesgue-measurable, its measure is denoted by $\vol(M)\in [0,\infty]$. The sphere in $\R^d$ is the set $\mathbb{S}^{d-1}:=\{x \in \R^d ~|~ \norm{x}_2=1\}$.
  
  As usual, $\GL_d(\R)$ is the set of $d\times d$ invertible matrices with entries from $\R$, and $O_d(\R)$ is the subgroup of orthogonal matrices. It is well known that any complex matrix $A$ has a singular value decomposition, where we denote the value of a smallest (respectively largest) singular value of $A$ by $\sigma_{\text{min}}(A)$ (respectively $\sigma_{\text{max}}(A)$). We denote the identity matrix in dimension $d$ by $I_d$. 
  
  For sequences $E,W: \N \to \R$, we write $W_N \in \mathcal{O}(E_N)$ if there exist $C>0$ and $N_0 \in \N$ such that $W_N \leq C \cdot E_N$ holds for all $N \in \N_{\geq N_0}$. We denote the support of a function $f: \R^d \to \C$ by $\supp(f):=\overline{\set{x \in \R^d ~|~ f(x)\neq 0}}$, where the closure is taken with respect to standard Euclidean topology on $\R^d$. For $A\in\GL_d(\R)$, $p \in \{1,2\}$, we also define $D_A^pf:=|\det(A)|^{\nicefrac{1}{p}} f(A \:\cdot)$. 
  Moreover, we set $f^*:=\overline{f(-\:\cdot)}$. By $L^p(\R^d)$, for $p\in [1,\infty)$, we refer to the standard Lebesgue space, i.e., the set of Lebesgue-measurable functions $f:\R^d \to \C$ such that \[\norm{f}_p^p:=\int_{\R^d} |f(x)|^p \dx<\infty,\] where functions that agree almost everywhere (a.e.) with respect to the Lebesgue-measure are identified. With the same identification, the space $L^\infty(\R^d)$ is the set of Lebesgue-measurable functions $f:\R^d \to \C$ that are bounded almost everywhere, i.e., $\norm{f}_\infty:=\inf\{c\in \R_{\geq 0}: |f(x)| \leq c \text{ a.e. } x \in \R^d\}<\infty$. The inner product of two functions $f,g$ from the Hilbert space $L^2(\R^d)$ is simply denoted by \[\ip{f,g}:=\int_{\R^d} f(x)\cdot \overline{g(x)} \dx.\]
  
  For $f \in L^1(\R^d)$, we define the Fourier transform of $f$ by
  \begin{align}\label{def: Fourier transform}
    \mathcal{F}f(\cdot):=\widehat{f}(\cdot):=\int_{\R^d} f(x) ~ e^{-2\pi i \ip{x, \cdot}} \dx.
  \end{align}
  The Fourier transform extends uniquely to a unitary automorphism of $L^2(\R^d)$ by this choice of normalization.
  Its inverse is the unique operator extension of the inverse Fourier transform given by 
  \[\mathcal{F}^{-1}f(\cdot):=\widehat{f}(-\:\cdot)=\int_{\R^d} f(x) ~ e^{2\pi i \ip{x, \cdot}} \dx.\]
  
  If $\omega:(0,\infty)\to (0,\infty)$ is a nondecreasing function, we define
  \begin{align}\label{def: weighted Sobolev space}
  \mathcal{F}L_\omega^2(\R^d):=\set{f \in L^2(\R^d)~\middle|~  \omega(\norm{\cdot}_2) \widehat{f} \in L^2(\R^d)}.
  \end{align}
  For $s>0$, setting $\omega_s(t):=(1+t^2)^{\frac{s}{2}}$, the fractional Sobolev space of class regularity $s$ is $H^s(\R^d):=\mathcal{F}L^2_{\omega_s}(\R^d)$. We define the $s$-fractional Log-Sobolev space by $H^s_{\log}(\R^d):=\mathcal{F}L^2_{\omega_{s,\log}}(\R^d)$, where $\omega_{s,\log}(t):=\ln^s(e+t)$. Furthermore, let $S(\R^d)$ denote the Schwartz space on $\R^d$. 
  For a Hilbert space $\h$, a countable set $\Lambda$, and $p \in [1,\infty)$, we define
  \begin{align}\label{def: vector-valued sequence space}
    \ell^p(\h):=\ell^p(\Lambda;\h):=\left\{f=(f_\lambda)_{\lambda \in \Lambda} \in \h^\Lambda~\middle|~ \sum_{\lambda \in \Lambda} \norm{f_\lambda}_{\h}^p<\infty\right\}.
  \end{align}
  If $p=2$, this is a Hilbert space, equipped with the inner product given by $\ip{f,g}:=\sum_{\lambda \in \Lambda} \ip{f_\lambda,g_\lambda}$ for $f,g \in \ell^2(\h)$ (with unconditional convergence of the series).
  
  \subsection{A brief review of the windowed scattering transform.}\label{sec: brief review}
  
  A scattering network computes cascades of convolutions followed by the application of the modulus as non-linearity, with no pooling between the layers. Given any $\psi \in L^1(\R^d)\cap L^2(\R^d)$, we denote by 
  \begin{align}\label{def: single step scattering propagator}
    U[\psi]:L^2(\R^d)\to L^2(\R^d), \quad f \mapsto |f*\psi|
  \end{align}
  the \textit{single step scattering propagator} associated to filtering with the function $\psi$ (which in this context is thus commonly called a \textit{filter}). Moreover, we refer to an ordered sequence $p=(\psi_1,\ldots,\psi_N)\in (L^1(\R^d)\cap L^2(\R^d))^N$ 
  as a \textit{path} of length $N \in \N$. Iterating the single step scattering propagators along this path yields the \textit{scattering propagator}
  \begin{align}\label{def: scattering propagator}
    U[p]:L^2(\R^d)\to L^2(\R^d), \quad f \mapsto U[\psi_N]\ldots U[\psi_1]f = |\ldots|f*\psi_1|\ldots*\psi_N|.
  \end{align}
  For notational convenience, we define the scattering propagator along the empty path $e$ (the unique path of length zero) to be the identity operator $U[e]:=\id_{L^2(\R^d)}$. 
  
  A designated function $\chi \in L^1(\R^d)\cap L^2(\R^d)$ induces a \textit{windowed scattering propagator} along a path $p$ 
  of length $N \in \N$ and window $\chi$ (referred to as output-generating filter), given by 
  \begin{align}\label{def: windowed scattering propagator}
    S[p;\chi]:L^2(\R^d)\to L^2(\R^d), \quad f \mapsto (U[p]f)*\chi = |\ldots|f*\psi_1|\ldots*\psi_N|*\chi.
  \end{align}
  By Young's convolution inequality, the operators $U[p]$ and $S[p;\chi]$ are well-defined.  
   
  Our main objective is to study the properties of the operators (Fig.~\ref{fig: restricted network architecture})
  \begin{align}\label{def: collection of scattering propagators}
    U[P]: L^2(\R^d) \to \ell^2(P;L^2(\R^d)), \quad f \mapsto \left(U[p]f\right)_{p \in P}
  \end{align}
  and
  \begin{align}\label{def: collection of windowed scattering propagators}
    S[P;\chi]: L^2(\R^d) \to \ell^2(P;L^2(\R^d)), \quad f \mapsto \left(S[p;\chi]f\right)_{p \in P}
  \end{align}
  for certain path sets $P$. 
  In particular, this requires the operators to be well-defined in the sense that, for any $f \in L^2(\R^d)$, the series $\sum_{p \in P} \norm{U[p]f}_2^2$ converges unconditionally. In this respect, let us now state our main requirement on the filters, which are the basic building blocks of the windowed scattering transform. 
  
  \begin{figure}[t!]
    \centering
    \captionsetup{width=.9\linewidth}
    \includegraphics[width=\textwidth]{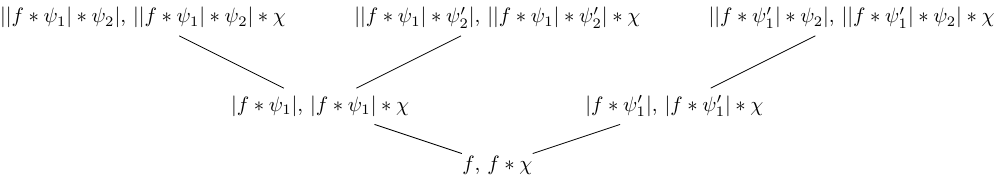}
    \caption{Illustration of the operators $U[P],S[P;\chi]$ for the path set $P=\{e,\psi_1, \psi_1^\prime, (\psi_1,\psi_2),(\psi_1,\psi_2^\prime),(\psi_1^\prime,\psi_2)\}$ and output-generating filter $\chi$.}
    \label{fig: restricted network architecture}
  \end{figure}
  
  Throughout this paper, we assume that $\mathfrak{F}:=\{\chi\} \cup \Psi \subset L^1(\R^d) \cap L^2(\R^d)$ is a countably infinite collection of functions so that the corresponding filter transform
  \begin{align}\label{filter transform}
    L^2(\R^d) \to \ell^2(\mathfrak{F};L^2(\R^d)), \ f \mapsto (f*\chi) \cup \left(f*\psi\right)_{\psi \in \Psi}
  \end{align}
  is an isometry. This is precisely the case if $\mathfrak{F}$ forms a semi-discrete Parseval frame \cite[Section 5.1.5]{mallat1999wavelet}, i.e., if for all $f \in L^2(\R^d)$,
  \begin{align}\label{ass: Parseval Frame condition}
    \norm{f*\chi}_2^2 + \sum_{\psi \in \Psi} \norm{f*\psi}_2^2 = \norm{f}_2^2.
  \end{align}
  We frequently use the fact (which follows directly from Parseval's theorem and the convolution theorem) that $\mathfrak{F}$ forms a semi-discrete Parseval frame if and only if $\mathfrak{F}$ satisfies the Littlewood-Paley condition
  \begin{align}\label{ass: Littlewood-Paley condition}
    |\widehat{\chi}(\xi)|^2+\sum_{\psi \in \Psi} |\widehat{\psi}(\xi)|^2=1 \quad \text{a.e. } \xi \in \R^d.
  \end{align}
  If any of the latter equivalent conditions holds, then every $f \in L^2(\R^d)$ admits a decomposition 
  \begin{align}
    f= f*\chi*\chi^* + \sum_{\psi \in \Psi} f*\psi*\psi^*.
  \end{align}
  
  \begin{remark}
    To simplify the notation, we do not assign the filters $\Psi$ an explicit (countably infinite) index set, unless we work in a more concrete setting. This precludes the use of multiple occurrences of the same filter, which is often allowed in the definition of the scattering transform. In principle, this can be accommodated, primarily at notational cost, by extending the subsequent discussion to multisets. To us the benefit of such an extension seems quite limited, which is why we have refrained from making the necessary adjustments. 
  \end{remark}
  
  We mainly consider the operators $U[P]$ and $S[P;\chi]$ on nonempty subsets $P$ of the set of all finite length scattering paths $\mathcal{P}_\Psi:=\bigcup_{N=0}^\infty \Psi^N$, where we set $\Psi^0:=\{e\}$. The \textit{windowed scattering transform associated with the filters $\mathfrak{F}$} is the operator
  \begin{align}\label{def: windowed scattering transform}
    \s[\mathfrak{F}]:=S[\mathcal{P}_\Psi;\chi]: L^2(\R^d) \to \ell^2(\mathcal{P}_\Psi;L^2(\R^d)), \quad f \mapsto \left(S[p;\chi]f\right)_{p \in \mathcal{P}_\Psi}.
  \end{align}
  In the interest of generality, we do not make further assumptions on the Parseval frame $\mathfrak{F}$ for now. Even in this very general setting, the windowed scattering transform $\s[\mathfrak{F}]$ (Fig.~\ref{fig: network architecture}) is always a well-defined operator into the Hilbert space $\ell^2(\mathcal{P}_\Psi;L^2(\R^d))$. We address this in more detail in Proposition \ref{prop: Energy decomposition and norm-preservation of S}.
  
  \begin{figure}[t!]
    \centering
    \includegraphics[width=\textwidth]{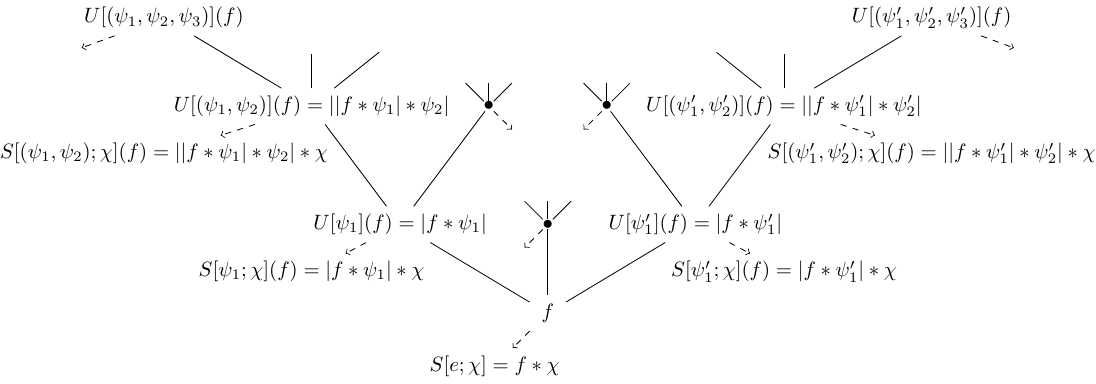}
    \caption{Indication of the architecture of a scattering network as described above, with filters $\psi_N,\psi_N^\prime \in \Psi$ corresponding to the $N$th network layer, $N \in \{1,2,3\}$. The function $\chi$ is the output-generating filter, which is (in our setup) the same across all layers. The residual energy $W_N(f)$ is the aggregated energy of the $N$th horizontal layer of the network, e.g., $W_2(f)=W_2[\Psi](f)=\norm{|f*\psi_1|*\psi_2}_2^2+\Vert|f*\psi_1^\prime|*\psi_2^\prime\Vert_2^2+\cdots$.}
    \label{fig: network architecture}
  \end{figure}
  
  Typically, the filters are concentrated in different areas in frequency domain, with $\chi$ taking on the special role of a low-pass filter and the remaining filters $\Psi$ being of high-pass nature. For the sake of concreteness and later use, let us give one particular example of an admissible collection $\mathfrak{F}$ in dimension $d=1$, where in this case $\mathfrak{F}$ is generated from a bandlimited wavelet.
  
  \begin{example}\label{ex: wavelet filters in dimension 1}
    Let $\psi, \phi \in L^1(\R)\cap L^2(\R)$ satisfy $\supp(\widehat{\psi}) \subseteq [\frac{1}{2},2]$ and 
    \begin{align*}
      |\widehat{\phi}(\xi)|^2+\sum_{j=1}^\infty \left(|\widehat{\psi}(2^{-j}\cdot \xi)|^2+|\widehat{\psi}(-2^{-j}\cdot \xi)|^2\right)=1, \quad \text{a.e. } \xi \in \R.
    \end{align*}
    Such a pair of functions arises, e.g., from the construction of an orthogonal wavelet basis, with father wavelet (scaling function) $\phi$ and mother wavelet $\psi$. Then, $\mathfrak{F}:=\{\chi\}\cup \{\psi_j ~|~ j \in \Z\setminus \{0\}\}$ is a semi-discrete Parseval frame, where $\chi:=\phi$ and
    \begin{align*}
      \psi_j:=
      \begin{cases}
        D^1_{2^j}\psi=2^j \psi(2^j\:\cdot) & \text{for } j \in \N \\
        D^1_{2^{|j|}}\psi(-\:\cdot)=2^{|j|} \psi(-2^{|j|}\:\cdot) &\text{for } j \in \Z_{<0}
      \end{cases}.
    \end{align*}
    The filter transform is in this case the associated wavelet transform.
  \end{example}
  
  The following proposition, which was first stated in (the proof of) \cite[Theorem 2.6]{mallat2012group} and later generalized in \cite[Proposition 1]{wiatowski2017mathematical} and \cite[Proposition 1]{wiatowski2017energy}, forms the basis for all further analysis in this paper.
  
  \begin{proposition}\label{prop: Energy decomposition and norm-preservation of S}
    Every $f \in L^2(\R^d)$ satisfies the energy decomposition, for all $N \in \N_0$, 
    \begin{align}\label{eq: energy decomp}
      \sum_{n=0}^{N-1} \norm{S[\Psi^n;\chi]f}_{\ell^2(\Psi^n;L^2(\R^d))}^2 + \norm{U[\Psi^N]f}_{\ell^2(\Psi^N;L^2(\R^d))}^2=\norm{f}_2^2.
    \end{align}
    Consequently, 
    \[\norm{\s[\mathfrak{F}]f}_{\ell^2(\mathcal{P}_\Psi;L^2(\R^d))}\leq \norm{f}_2,\]
    which ensures the well-definedness of the operator $\s[\mathfrak{F}]: L^2(\R^d) \to \ell^2(\mathcal{P}_\Psi;L^2(\R^d))$.
  
    Moreover,
    \begin{align}\label{eq: lim W_N(f) to 0}
      \lim_{N \to \infty} \norm{U[\Psi^N]f}_{\ell^2(\Psi^N;L^2(\R^d))}^2=0
    \end{align}
    holds if and only if  
    \begin{align}\label{eq: norm-preservation of S if W_N(f) to 0}
      \norm{\s[\mathfrak{F}]f}_{\ell^2(\mathcal{P}_\Psi;L^2(\R^d))}^2= \norm{f}_2^2.
    \end{align}
  \end{proposition}
  
  The energy decomposition identity \eqref{eq: energy decomp} quantifies the loss of information about the original input signal $f$ if the scattering network is restricted to its first $N$ layers. Motivated by this important relationship, let us therefore define the quantity, for any given $\emptyset \neq \Phi \subseteq \mathfrak{F}$ and $N \in \N_0$, 
  \begin{align}\label{def: energy remainder}
    W_N[\Phi](f):=\norm{U[\Phi^N]f}_{\ell^2(\Phi^N;L^2(\R^d))}^2=\sum_{p \in \Phi^N} \norm{U[p](f)}_2^2.
  \end{align}
  
  Our main objective in this paper is to study the asymptotic behavior of the quantity $W_N[\Psi](f)$ as $N\to \infty$, which, as we show, distinctively depends on the interplay between the choice of filters $\Psi$ and the input signal $f$. 
  If there is no ambiguity about the underlying filters $\Psi$, we sometimes only write $W_N(f)$ to refer to $W_N[\Psi](f)$.
  
  As already pointed out in Section \ref{sec: Motivation and related work}, there are several sufficient criteria for $\mathfrak{F}$ available that guarantee the asymptotic energy decay \eqref{eq: lim W_N(f) to 0}. More precisely, \cite[Theorem 3.1]{waldspurger2017exponential} and \cite[Theorem 1, Theorem 2]{wiatowski2017energy} state an upper bound on $W_N[\Psi](f)$, for all $f \in L^2(\R^d)$ and $N \in \N_{\geq 2}$, which is of the type
  \begin{align}
    W_N[\Psi](f) \leq \int_{\R^d} |\widehat{f}(\xi)|^2 \cdot K_N(\xi) \dxi
  \end{align}
  for a family of integral kernels $(K_N)_{N \in \N_{\geq 2}}$ (independent of $f$) that satisfy the following properties:
  \begin{itemize}
    \item For all $N \in \N_{\geq 2}$, $K_N$ is continuous.
    \item For all $N \in \N_{\geq 2}$, $0\leq K_{N+1} \leq K_N \leq 1$.
    \item $\lim_{N \to \infty} K_N(\xi)=0$ holds pointwise for all $\xi \in \R^d$.
  \end{itemize}
  By Dini's theorem \cite[Theorem 7.3]{dibenedetto2002real}, this upper integral bound entails, for all $f \in L^2(\R^d)$, the convergence \eqref{eq: lim W_N(f) to 0} (cf. \cite[Proposition 2]{wiatowski2017energy}). 
  In particular, the scattering transform $\s[\mathfrak{F}]$ is norm-preserving in these settings by Proposition \ref{prop: Energy decomposition and norm-preservation of S}. Specifically the assumptions of \cite[Theorem 1]{wiatowski2017energy} allow for some flexibility in the choice of the filters while guaranteeing norm-preservation. Clearly, this implies that the single signal that produces zero-features in every layer by means of the scattering transform is the trivial null-signal $f=0$. 
  
  In order to produce many informative (in the sense of nontrivial) features by means of the scattering transform, one typically chooses an output-generating filter $\chi$ that is of low-pass type. By the Littlewood-Paley condition \eqref{ass: Littlewood-Paley condition}, the remaining filters $\Psi$ are automatically of high-pass nature. Let us briefly illustrate how informativeness of $\s[\mathfrak{F}]$ can be achieved by imposing a low-pass condition on $\chi$: Let $0\neq g,\chi \in L^1(\R^d)\cap L^2(\R^d)$ and suppose that $\widehat{\chi}(0)\neq 0$. By continuity of $\widehat{\chi}$ and $\mathcal{F}(|g|)$, and since $\mathcal{F}(|g|)(0)=\norm{g}_1 >0$, there exist $\varepsilon,\delta>0$ such that $|\widehat{\chi}|\cdot|\mathcal{F}(|g|)|\geq \varepsilon$ on $B_\delta(0)$. Thus, we have
  \[\norm{|g|*\chi}_2^2=\norm{\mathcal{F}(|g|)\cdot \widehat{\chi}}_2^2\geq \int_{B_\delta(0)} |\mathcal{F}(|g|)(\xi)|^2 \cdot |\widehat{\chi}(\xi)|^2 \dxi \geq \varepsilon^2 \cdot \vol(B_\delta(0))>0.\]
  Hence, any feature of this type is nontrivial. The assumption that $\widehat{\chi}(0)\neq 0$ is of vital importance for the argument, as is shown in \cite[Appendix A]{wiatowski2017energy} by the example of a scattering network (feature extractor) with nontrivial kernel.
  
  For the remainder of this section, we collect some basic auxiliary statements about the windowed scattering transform that will prove useful later. We include their proofs for self-containment.
  
  \begin{lemma}\label{lem: U is norm-preserving if we consider the whole Parseval frame}
    Let $N \in \N$. Then $U[\mathfrak{F}^N]:L^2(\R^d)\to \ell^2(\mathfrak{F};L^2(\R^d))$ is well-defined and norm-preserving, i.e., for all $f \in L^2(\R^d)$,
    \[W_N[\mathfrak{F}](f)=\sum_{p \in \mathfrak{F}^N}\norm{U[p]f}_2^2=\norm{f}_2^2.\]
  \end{lemma}
  \begin{proof}
    This is easily shown by induction: The base case $N=1$ is exactly one of the characterizations of a semi-discrete Parseval frame, hence true for $\mathfrak{F}$. The induction step is established as follows. Suppose that the statement is true for some $N \in \N$. Then, for all $f \in L^2(\R^d)$, we have
    \begin{align*}
      W_{N+1}[\mathfrak{F}](f)=\sum_{\psi \in \mathfrak{F}} W_N[\mathfrak{F}](|f*\psi|)= \sum_{\psi \in \mathfrak{F}} \norm{|f*\psi|}_2^2 = \sum_{\psi \in \mathfrak{F}} \norm{f*\psi}_2^2=\norm{f}_2^2.
    \end{align*}
  \end{proof}
  
  Despite its simplicity and brevity, the argument illustrates how the underlying inductive structure of the scattering network can be exploited in proofs. The following Proposition (cf. \cite[Proposition 2.5]{mallat2012group}) states the non-expansiveness of the operators $U$ and $S$, which is also due to their inductive definition. Considering the desirable properties of the scattering transform as a feature extractor, this property can be considered as stability against additive noise.
  
  \begin{proposition}\label{prop: U, S nonexpansive}
    Let $\emptyset \neq P \subseteq \Psi^N$ for some $N \in \N_0$. Then, for all $f,g \in L^2(\R^d)$,
    \begin{align*}
      \norm{S[P;\chi]f-S[P;\chi]g}_{\ell^2(P;L^2(\R^d))} \leq \norm{U[P]f-U[P]g}_{\ell^2(P;L^2(\R^d))}\leq \norm{f-g}_2.
    \end{align*}
  
    Furthermore, if $\s[\mathfrak{F}]=S[\mathcal{P}_\Psi;\chi]$ is norm-preserving, then it is also nonexpansive.
  \end{proposition}
  
  \begin{proof}
    By the Littlewood-Paley condition \eqref{ass: Littlewood-Paley condition}, we have $\norm{\widehat{\chi}}_\infty \leq 1$. Thus, for all $p \in P$, and all $f,g \in L^2(\R^d)$, 
    \[\norm{S[p;\chi]f-S[p;\chi]g}_2=\norm{(U[p]f-U[p]g)*\chi}_2 \leq \norm{U[p]f-U[p]g}_2,\]
    which proves the first inequality. 
    
    Observe that it suffices to prove the second inequality for the case $P=\Psi^N$, which we do by induction on $N\in \N_0$. The base case $N=0$ is trivial. For the purpose of the induction step, suppose that the hypothesis is true for some $N \in \N_0$. 
    By the reverse triangle inequality, we have, for all $f,g \in L^2(\R^d)$,
    \begin{align*}
      \lvert|U[\Psi^{N+1}]f-&U[\Psi^{N+1}]g\rvert|_{\ell^2(\Psi^{N+1};L^2(\R^d))}^2 \\
      &= \sum_{p \in \Psi^N}\sum_{\psi \in \Psi} \norm{|(U[p]f)*\psi|-|(U[p]g)*\psi|}_2^2 \\
      &\leq \sum_{p \in \Psi^N}\sum_{\psi \in \Psi} \norm{|(U[p]f-U[p]g)*\psi|}_2^2\\
      &=\sum_{p \in \Psi^N}\sum_{\psi \in \Psi} \norm{(U[p]f-U[p]g)*\psi}_2^2 \\
      &= \sum_{p \in \Psi^N} (\norm{U[p]f-U[p]g}_2^2-\norm{S[p;\chi]f-S[p;\chi]g}_2^2) \\
      &= \norm{U[\Psi^{N}]f-U[\Psi^{N}]g}_{\ell^2(\Psi^{N};L^2(\R^d))}^2-\norm{S[\Psi^{N};\chi]f-S[\Psi^{N};\chi]g}_{\ell^2(\Psi^{N};L^2(\R^d))}^2,
    \end{align*}
    where the second last step is due to $\mathfrak{F}$ being a frame. Applying the induction hypothesis to the right-hand side of the latter inequality concludes this part of the proof. 
  
    Finally, if $\s[\mathfrak{F}]$ is norm-preserving, the non-expansiveness of $\s[\mathfrak{F}]$ follows directly from the above estimates (which were derived independent of the induction hypothesis) using a telescoping series argument. In fact, by Proposition \ref{prop: Energy decomposition and norm-preservation of S}, we have $\lim_{N \to \infty} W_N[\Psi](h)=0$ for all $h \in L^2(\R^d)$, which entails that, for all $f,g \in L^2(\R^d)$,
    \[\lim_{N \to \infty} \norm{U[\Psi^{N}]f-U[\Psi^{N}]g}_{\ell^2(\Psi^{N};L^2(\R^d))}^2=0.\]
    Consequently,
    \begin{align*}
      \lvert|\s[\mathfrak{F}]f-&\s[\mathfrak{F}]g\rvert|_{\ell^2(\mathcal{P}_\Psi;L^2(\R^d))}^2 \\
      &= \sum_{N=0}^\infty \norm{S[\Psi^{N};\chi]f-S[\Psi^{N};\chi]g}_{\ell^2(\Psi^{N};L^2(\R^d))}^2 \\
      &\leq \sum_{N=0}^\infty \left(\norm{U[\Psi^{N}]f-U[\Psi^{N}]g}_{\ell^2(\Psi^{N};L^2(\R^d))}^2-\norm{U[\Psi^{N+1}]f-U[\Psi^{N+1}]g}_{\ell^2(\Psi^{N+1};L^2(\R^d))}^2\right)\\
      &= \norm{f-g}_2^2.
    \end{align*}
  \end{proof}
  
  We exploit the non-expansiveness of $U$ to derive some useful estimates concerning the energy remainder.
  
  \begin{lemma}\label{lem: Lipschitz-type bounds for W_N} 
    Let $\emptyset \neq \Phi \subseteq \Psi$, and let $N \in \N$. Then, we have, for all $f,g \in L^2(\R^d)$, 
    \begin{myenum}
      \item[a)] $|W_N[\Phi](f)-W_N[\Phi](g)| \leq \sqrt{2} \cdot \norm{f-g}_2 \cdot \sqrt{W_N[\Phi](f)+W_N[\Phi](g)}$,
      \item[b)] $W_N[\Phi](f)\geq \frac{1}{2} \cdot W_N[\Phi](g)-\norm{f-g}_2^2.$
    \end{myenum}
  \end{lemma}
  \begin{proof}
    Let $f,g \in L^2(\R^d)$. Applying the reverse triangle inequality yields, for any $p \in \Phi^N$,
    \begin{align*}
      |\norm{U[p]f}_2^2-\norm{U[p]g}_2^2|&=|\norm{U[p]f}_2-\norm{U[p]g}_2|\cdot (\norm{U[p]f}_2+\norm{U[p]g}_2) \\
      &\leq \norm{U[p]f-U[p]g}_2 \cdot (\norm{U[p]f}_2+\norm{U[p]g}_2). 
    \end{align*}
    Together with the Cauchy-Schwarz inequality and the non-expansiveness of $U[\Phi^N]$, we obtain 
    \begin{align*}
      |W_N[\Phi](f)-W_N[\Phi](g)|^2 &= \left|\sum_{p \in \Phi^N} \norm{U[p]f}_2^2-\norm{U[p]g}_2^2\right|^2 \\
      &\leq \left(\sum_{p \in \Phi^N} \norm{U[p]f-U[p]g}_2 \cdot \left(\norm{U[p]f}_2+\norm{U[p]g}_2\right)\right)^2 \\
      &\leq \left(\sum_{p \in \Phi^N} \norm{U[p]f-U[p]g}_2^2\right) \cdot \left(\sum_{p \in \Phi^N} \left(\norm{U[p]f}_2+\norm{U[p]g}_2\right)^2\right)\\
      &\leq \norm{f-g}_2^2 \cdot \sum_{p \in \Phi^N} \left(2\norm{U[p]f}_2^2+ 2\norm{U[p]g}_2^2\right) \\
      &= 2 \norm{f-g}_2^2 \cdot (W_N[\Phi](f)+W_N[\Phi](g)).
    \end{align*}
    Taking square roots on both sides of the latter inequality proves a).
  
    We again use the non-expansiveness of $U[\Phi]$ and find that
    \begin{align*}
      W_N[\Phi](g)&=\norm{U[\Phi^N](g)-U[\Phi^N](f)+U[\Phi^N](f)}_{\ell^2(\Phi^N;L^2(\R^d))}^2 \\
      &\leq \left(\norm{U[\Phi^N](g)-U[\Phi^N](f)}_{\ell^2(\Phi^N;L^2(\R^d))}+\norm{U[\Phi^N](f)}_{\ell^2(\Phi^N;L^2(\R^d))}\right)^2 \\
      &\leq \left(\norm{g-f}_2+\norm{U[\Phi^N](f)}_{\ell^2(\Phi^N;L^2(\R^d))}\right)^2 \\
      &\leq 2 \norm{f-g}_2^2 + 2 \norm{U[\Phi^N](f)}_{\ell^2(\Phi^N;L^2(\R^d))}^2 \\
      &= 2 \norm{f-g}_2^2 + 2 W_N[\Phi](f).
    \end{align*}
    By rearranging, we conclude the proof of b).
  \end{proof}
  
  It is now easy to derive the continuity of $W_N[\Phi]$, which we use to prove the density result, Proposition \ref{prop: density of adversarial examples}, in Section 2. 
  \begin{corollary}\label{cor: W_N is continuous}
    Let $\emptyset \neq \Phi \subseteq \Psi$, and let $N \in \N$. Then, $W_N[\Phi]:L^2(\R^d)\to \R$ is continuous.
  \end{corollary}
  \begin{proof}
    As the composition of continuous maps, $W_N[\Phi]= \lVert \text{} \cdot \text{} \rVert _{
    \ell ^{2}(\Phi ^{N};L^{2}({\mathbb{R}}^{d}))}^{2} \circ U[\Phi ^{N}]$ is continuous. Alternatively, one can use the estimate from Lemma~\ref{lem: Lipschitz-type bounds for W_N}a) to derive the continuity of $W_N[\Phi]$ more explicitly: Continuity at $f=0$ is obvious, since we have, for all $g \in L^2(\R^d)$, $W_N[\Phi](g)\leq \norm{g}_2^2$. Now, fix $\varepsilon>0$ and $f \in L^2(\R^d)\setminus\{0\}$. Define \[\delta=\delta(f,\varepsilon):=\min\left\{\frac{\varepsilon}{\sqrt{10} \norm{f}_2},\norm{f}_2\right\}>0.\] 
    Applying Lemma \ref{lem: Lipschitz-type bounds for W_N}a) to $g \in L^2(\R^d)$ with $\norm{f-g}_2<\delta$ gives
    \begin{align*}
      |W_N[\Phi](f)-W_N[\Phi](g)|^2 &\leq 2 \norm{f-g}_2^2 \cdot(W_N[\Phi](f)+W_N[\Phi](g))\\
      &\leq 2 \delta^2 \cdot (\norm{f}_2^2+\norm{g}_2^2) \\
      &\leq 2 \delta^2 \cdot (\norm{f}_2^2+(\norm{f}_2+\delta)^2) \\
      &\leq 10 \delta^2 \cdot \norm{f}_2^2 \leq \varepsilon^2.
    \end{align*}
  \end{proof}

  We conclude this section with a lemma describing the interplay of the scattering propagator with dilations, which is similar to covariance (cf. \cite[Equation (20)]{mallat2012group}). Despite its elementary proof, this property forms the basis for our main result in Section \ref{sec: Slow Scattering Propagation}.  
  \begin{lemma}\label{lemma: commuting U[p] with D_A}
    Let $N \in \N$, let $p\in (L^1(\R^d)\cap L^2(\R^d))^N$, and let $A \in \GL_d(\R)$.
    Then, for all $f \in L^2(\R^d)$,
    \begin{align*}
      U[p]D_{A}^1f=D_{A}^1U[D_{A^{-1}}^1p]f,
    \end{align*}
    where $D_{A^{-1}}^1p:=(D_{A^{-1}}^1p_1,\ldots,D_{A^{-1}}^1p_N)$.
    Moreover, if $\emptyset \neq \Phi \subseteq \mathfrak{F}$, then we have, for all $f \in L^2(\R^d)$,
    \[W_N[\Phi](D_{A}^2 f)=W_N[D_{A^{-1}}^1\Phi](f).\]
  \end{lemma}
  
  \begin{proof}
    Let $f \in L^2(\R^d)$, and let $g \in L^1(\R^d)\cap L^2(\R^d)$. A direct computation shows that 
    \[(D_A^1f)*g=D_A^1(f*(D_{A^{-1}}^1g)).\]
    The modulus operator 
    \[M:L^2(\R^d)\to L^2(\R^d), \quad h \mapsto |h|\] 
    commutes with $D_A^1$, which proves $U[g]D_A^1f=D_A^1U[D_{A^{-1}}^1g]f$. Iterating this relation along the path $p$ concludes the first identity.
    
    Concerning the second, we note that $D_A^1= |\det(A)|^{\frac{1}{2}}\cdot D_A^2$, which gives $U[p]D_A^2=D_A^2U[D_{A^{-1}}^1p]$. Since $D_A^2$ is unitary on $L^2(\R^d)$, summing over all paths in $\Phi^N$ entails the second identity.
  \end{proof}
  \section{Slow Scattering Propagation}\label{sec: Slow Scattering Propagation}
  
  In this section, we show that energy propagation can be arbitrarily slow in scattering networks that employ filters $\mathfrak{F}=\{\chi\}\cup \Psi$ with an underlying structure similar to those generated by wavelets.
  
  As some of the details of this section are quite technical, we begin by illustrating the conceptual idea behind them. In light of this, let us consider the situation from Example \ref{ex: wavelet filters in dimension 1} again, i.e., suppose that the high-pass filters $\Psi=\{\psi_j ~|~ j \in \Z\setminus \{0\}\}$ are generated by dilations of $\psi \in L^1(\R)\cap L^2(\R)$, $\supp(\widehat{\psi})\subseteq [\frac{1}{2},2]$, according to
  \begin{align*}
    \psi_j:=
    \begin{cases}
      D^1_{2^j}\psi=2^j \psi(2^j\:\cdot) & \text{for } j \in \N \\
      D^1_{2^{|j|}}\psi(-\:\cdot)=2^{|j|} \psi(-2^{|j|}\:\cdot) &\text{for } j \in \Z_{<0}
    \end{cases}.
  \end{align*}
  We want to show that for any nonincreasing null-sequence $E\in \R_{>0}^\N$, there exists a function $f_E\in L^2(\R)$ that satisfies the following conditions:
  \begin{myenum}
    \item[a)] We have $\norm{f_E}_2^2=E_1$.
    \item[b)] For all $N \in \N$, we have that $W_N[\Psi](f_E)\geq \frac{1}{2}\cdot E_N$.
  \end{myenum}
  
  \begin{proof sketch}
  Define 
  \[(a_k)_{k \in \N}:=\left(\sqrt{E_k-E_{k+1}}\right)_{k \in \N},\] 
  and note that a telescoping argument gives 
  \[\lim_{N \to \infty} \sum_{k=1}^N a_k^2 = \lim_{N \to \infty} E_1-E_{N+1}=E_1.\] 
  Fix any $f_0 \in L^2(\R)$ that satisfies $\norm{f_0}_2=1$ and $\supp(\widehat{f_0})\subseteq [1,2]$. We will later (see Corollary \ref{cor: slow propagation in case of a single expansive matrix} and Corollary \ref{cor: Wavelet decay rates}) show that, for all $N \in \N$,
  \[\lim_{m \to \infty} W_N\left(f_m\right)=1,\]
  where $f_m:=2^{m/2}f_0(2^m \,\cdot\,)$, $m \in \N$. This is a consequence of Lemma \ref{lemma: commuting U[p] with D_A}. 
  \begin{figure}[t!]
    \centering
    \begin{minipage}{0.999\textwidth}
      \centering
      \resizebox{0.999\textwidth}{!}{\input{f_0_Fourier_dilations_large_font.pgf}} 
    \end{minipage}\hfill
    \captionsetup{width=.78\linewidth}
    \caption{Fourier transform of a sample function $f_0$, and of its dilations $f_2$, $f_4$, $f_6$, illustrating the frequency separation between the individual components of $f_E$ resulting from the dilation. Frequency separation is key to our construction of adversarial signals. In this illustration, $f_0=\frac{h}{\norm{h}_2}$, and $\widehat{h}(\xi)=\mathds{1}_{(1,2)} \cdot \exp\left(-\frac{4}{1-(3-2\xi)^2}\right)$.}
    \label{fig: f_0 Fourier dilations}
  \end{figure}

  Therefore, we can find a strictly increasing sequence $(m_k)_{k \in \N} \subseteq \N$ such that, for all $k \in \N$, $m_{k+1}\geq 2+m_{k}$ and $W_k(f_{m_k})\geq \frac{1}{2}$. We claim that 
  \[f_E:=\sum_{k=1}^\infty a_k f_{m_k} \in L^2(\R)\]
  does the job (Figs.~\ref{fig: f_0 Fourier dilations},\ref{fig: illustr of proof principle}). To this end, first note that, for all $m \in \N$, $\supp(\widehat{f_m})\subseteq [2^m,2^{m+1}]$. In particular, the $(f_m)_{m \in \N}$ are orthonormal, hence $\norm{f_E}_2^2=\sum_{k=1}^\infty |a_k|^2=E_1$. Moreover, since 
  \[\supp(\widehat{\psi_j})\subseteq [\sgn(j)\cdot 2^{|j|-1},\sgn(j)\cdot 2^{|j|+1}],\] 
   \begin{figure}[h!]
  	\centering
  	\begin{minipage}{0.499\textwidth}
  		\centering
  		\resizebox{0.999\textwidth}{!}{\input{Re_f_E_zoomed_out_large_font.pgf}} 
  	\end{minipage}\hfill
  	\begin{minipage}{0.499\textwidth}
  		\centering
  		\resizebox{0.999\textwidth}{!}{\input{Im_f_E_zoomed_out_large_font.pgf}}
  	\end{minipage}
  	\\
  	\begin{minipage}{0.499\textwidth}
  		\centering
  		\resizebox{0.999\textwidth}{!}{\input{Re_f_E_zoomed_in_large_font.pgf}} 
  	\end{minipage}\hfill
  	\begin{minipage}{0.499\textwidth}
  		\centering
  		\resizebox{0.999\textwidth}{!}{\input{Im_f_E_zoomed_in_large_font.pgf}}
  	\end{minipage}
  	\captionsetup{width=0.9\linewidth}
  	\caption{Approximation of the real part (left), and imaginary part (right) of $f_E$, where $E=(\nicefrac{1}{k})_{k\in \N}$, and $f_0=\mathcal{F}^{-1}(\mathds{1}_{[1,\frac{3}{2}]}-\mathds{1}_{[\frac{3}{2},2]})$. To avoid numerical issues arising from numbers being close to zero, and for the sake of illustration, we have explicitly chosen $m_k=2k$, $k \in \N$. The top figures indicate the global behavior of $f_E$, while the zoomed-in figures below suggest the wild local behavior of our semi-explicit constructions.}\label{fig: illustr of proof principle}
  \end{figure}
  we have
  \begin{align*}
    f_E*\psi_j=
    \begin{cases}
      a_k f_{m_k}*\psi_{m_k} & \text{if } j=m_k \text{ for some } k \in \N \\
      a_k f_{m_k}*\psi_{m_k+1} & \text{if } j=m_k+1 \text{ for some } k \in \N \\
      0 &\text{else}
    \end{cases}, \quad j \in \Z\setminus \{0\}.
  \end{align*}
  Altogether, we obtain, for all $N \in \N$,\enlargethispage*{15mm}
  \begin{align*}
    W_{N}(f_E)&=\sum_{j \in \Z \setminus \{0\}} W_{N-1}(|f_E*\psi_j|) \\
    &=\sum_{k=1}^\infty \left(W_{N-1}(|a_kf_{m_k}*\psi_{m_k}|) + W_{N-1}(|a_kf_{m_k}*\psi_{m_k+1}|)\right) \\
    &=\sum_{k=1}^\infty |a_k|^2 \cdot \left(W_{N-1}(|f_{m_k}*\psi_{m_k}|) + W_{N-1}(|f_{m_k}*\psi_{m_k+1}|)\right)\\
    &= \sum_{k=1}^\infty |a_k|^2 \cdot W_{N}(f_{m_k})
    \geq \sum_{k=N}^\infty |a_k|^2 \cdot W_{k}(f_{m_k}) 
    \geq \frac{1}{2} \cdot \sum_{k=N}^\infty |a_k|^2 = \frac{1}{2} \cdot E_N. &\square
  \end{align*}
  \end{proof sketch}
  
  To summarize, the central idea of the proof is as follows: Given $\delta>0$, construct a sequence of signals $(f_{m_k})_{k \in \N} \subseteq L^2(\R)$ such that (i) the energy remainder $W_N$ behaves (super-)additively due to separation in the frequency domain and (ii) the energy decays slowly for $f_{m_k}$ across the first $k$ layers of the scattering network, according to $W_k(f_{m_k})\geq \delta$. Weighting the signals with an $\ell^2$-sequence ensures square-integrability of $f_E$. The energy that is contained in the scattering layers of $f_E$ of depth greater than $N-1$ (i.e., $W_N(f_E)$) is controlled by the weighted energy remainders of the signals $(f_{m_k})_{k \geq N}$. This results in a slow propagation of energy across all network layers. The speed of convergence is bounded from below by the speed of convergence of the $\ell^2$-series of the chosen weights.  
  
  In order to generalize this idea to larger classes of filter banks and higher dimensions, we need to establish a much more general version of super-additivity for the family of nonlinear operators $(W_N:L^2(\R^d)\to \R)_{N \in \N}$. In the proof sketch above, we have implicitly used their additivity on separated signals (cf. \cite[Lemma 2.3]{nicola2023stability}). However, if the generator $\psi$ fails to be bandlimited, we cannot generally expect exact additivity on separated signals. Moreover, we would like to allow for more flexibility in the choice of our filter bank. The following lemma establishes an \textit{approximate} version of super-additivity for the operators $(W_N)_{N \in \N}$ if the signals are nearly separated in Fourier domain. We use the high-pass filters $\Psi$ to measure the degree of separation.
  
  \begin{lemma}\label{lem: approx additivity}
    Let $F=(f_k)_{k\in \N_0} \subset L^2(\R^d)$, let $(\eta_k)_{k \in \N_0} \in \R_{>0}^{\N_0}$, and let $a=(a_k)_{k \in \N_0}\in \ell^2(\N_0;\C)$. Assume that there exist finite sets $\Psi_k\subseteq \Psi$, $k \in \N_0$, such that the following hold:
    \begin{myenum}
      \item[(i)] The series $\sum_{k,j=0}^\infty\left|a_ka_j \ip{f_k,f_j}\right|$ converges.
      \item[(ii)] The series $\sum_{k=0}^\infty \varepsilon_k$ converges, where $\varepsilon_k:=(k+1) \cdot \eta_k+\sum_{j=k+1}^\infty \eta_j$.
      \item[(iii)] We have $\Psi_{k_1}\cap \Psi_{k_2}=\emptyset$ whenever $k_1\neq k_2$. 
      \item[(iv)] We have, for every $k \in \N_0$,
      \begin{align*}
        \sum_{\psi \in \Psi\setminus\Psi_k} \norm{f_k*\psi}_2^2 \leq \eta_k.
      \end{align*}  
      \item[(v)] We have, for every $k \in \N_0$ and $0\leq j < k$,
      \begin{align*}
        \sum_{\psi \in \Psi_k} \norm{f_j*\psi}_2^2 \leq \eta_k.
      \end{align*}
    \end{myenum}
  
    Then, the series 
    \[f:=\sum_{k=0}^\infty a_k f_k\] 
    converges in $L^2(\R^d)$ with $\norm{f}_2^2\leq \sum_{k,j=0}^\infty\left|a_ka_j \ip{f_k,f_j}\right|$, and we have, for every $N \in \N$, $n \in \N_0$, 
    \begin{align}\label{inequ: lower bound for W_N(f) (technical)}
      W_N\left(f\right)\geq \sum_{k=n}^\infty \left(\frac{|a_k|^2}{2} \cdot W_{N}(f_k) - 2 \norm{a}_{\ell^2(\C)}^2 \cdot \varepsilon_k\right).
    \end{align}
  
    Moreover, suppose that in addition to our standing assumptions (i)-(v) there exists $\delta>0$ such that
    \begin{myenum}
      \item[(vi)] for all $k \in \N$, \[W_k(f_k)\geq 4\delta.\]
      \item[(vii)] for all $k \in \N$, \[a_k=0 \qquad \text{or} \qquad \varepsilon_k\leq \frac{\delta \cdot |a_k|^2}{2\norm{a}_{\ell^2(\C)}^2}.\]
    \end{myenum}
    Then we even have, for all $N \in \N$,
    \begin{align}\label{inequ: lower bound for W_N(f) (simplified)}
      W_N(f)\geq \delta \cdot \sum_{k=N}^\infty |a_k|^2.
    \end{align}
  \end{lemma}
  
  \begin{proof}
    The series defining $f$ converges if 
    \[\norm{\sum_{k=n}^m a_k f_k}_2 \xrightarrow{n,m \to \infty} 0.\] 
    Expanding the latter norm squared yields
    \begin{align}\label{inequ: (2)}
      \norm{\sum_{k=n}^m a_k f_k}_2^2 = \sum_{k,j=n}^m a_k \overline{a_j} \cdot \ip{f_k,f_j} \leq \sum_{k,j=n}^\infty \left|a_k a_j \cdot \ip{f_k,f_j}\right|.
    \end{align}
    Being the remainder of a convergent series by assumption (i), the right-hand side of \eqref{inequ: (2)} converges to $0$, so that $f$ is in fact a well-defined element of $L^2(\R^d)$. Setting $n=0$ in \eqref{inequ: (2)}, we also obtain the claimed bound for the (squared) norm of $f$ as $m \to \infty$.
  
    Let us now turn to the proof of \eqref{inequ: lower bound for W_N(f) (technical)}. Here, we first note that the recursive nature of $W_{N}$ and Lemma \ref{lem: Lipschitz-type bounds for W_N} b) imply
    \begin{align}\label{inequ: (1)}
      W_{N}(f) &= \sum_{\psi \in \Psi} W_{N-1}(|f*\psi|) \nonumber\\
      &\geq \sum_{k=n}^\infty \sum_{\psi \in \Psi_k}W_{N-1}(|f*\psi|)\nonumber\\
      &\geq \sum_{k=n}^\infty \sum_{\psi \in \Psi_k} \left(\frac{1}{2} \cdot W_{N-1}(|(a_kf_k)*\psi|)-\norm{|f*\psi|-|(a_kf_k)*\psi|}_2^2\right) \nonumber\\
      &\geq \sum_{k=n}^\infty \sum_{\psi \in \Psi_k} \left(\frac{|a_k|^2}{2} \cdot W_{N-1}(|f_k*\psi|)-\norm{(f-a_kf_k)*\psi}_2^2\right) \nonumber\\
      &= \sum_{k=n}^\infty \left(\frac{|a_k|^2}{2} \sum_{\psi \in \Psi_k}  W_{N-1}(|f_k*\psi|) - \text{Rest}_k\right).
    \end{align}
    As a consequence of our assumptions (iii)-(v), the remainder is bounded by
    \begin{align*}
      \text{Rest}_k &= \sum_{\psi \in \Psi_k} \norm{\left(\sum_{j=0,j\neq k}^\infty a_j f_j\right)*\psi}_2^2 \\
      &\leq \sum_{\psi \in \Psi_k} \left(\sum_{j=0,j\neq k}^\infty |a_j| \cdot \norm{f_j*\psi}_2\right)^2 \\
      &\leq 2 \sum_{\psi \in \Psi_k} \left[\left(\sum_{j=0}^{k-1} |a_j|\cdot \norm{f_j*\psi}_2\right)^2 + \left(\sum_{j=k+1}^{\infty} |a_j|\cdot \norm{f_j*\psi}_2\right)^2\right] \\
      &\leq 2 \norm{a}_{\ell^2(\C)}^2 \cdot \sum_{\psi \in \Psi_k} \left(\sum_{j=0}^{k-1} \norm{f_j*\psi}_2^2 + \sum_{j=k+1}^{\infty} \norm{f_j*\psi}_2^2\right)\\
      &\leq 2 \norm{a}_{\ell^2(\C)}^2 \cdot \left(k \cdot \eta_k + \sum_{j=k+1}^{\infty} \eta_j\right).
    \end{align*}
    Moreover, for every $k \geq n$,
    \begin{align*}
      \sum_{\psi \in \Psi_k} W_{N-1}(|f_k*\psi|)
      &= \sum_{\psi \in \Psi} W_{N-1}(|f_k*\psi|)- \sum_{\psi \in \Psi \setminus \Psi_k} W_{N-1}(|f_k*\psi|) \\
      &= W_{N}(f_k)-\sum_{\psi \in \Psi \setminus \Psi_k} W_{N-1}(|f_k*\psi|)\\
      &\geq W_{N}(f_k)-\eta_k.
    \end{align*}
    Substituting this and the bound for $\text{Rest}_k$ into \eqref{inequ: (1)} yields
    \begin{align*}
      W_{N}(f)&\geq \sum_{k=n}^\infty \left(\frac{|a_k|^2}{2} \sum_{\psi \in \Psi_k}  W_{N-1}(|f_k*\psi|) - \text{Rest}_k\right) \\
      &\geq \sum_{k=n}^\infty \left[\frac{|a_k|^2}{2} \cdot \left(W_{N}(f_k)-\eta_k\right) - 2 \norm{a}_{\ell^2(\C)}^2 \cdot \left(k\cdot \eta_k + \sum_{j=k+1}^{\infty} \eta_j\right)\right] \\
      &\geq \sum_{k=n}^\infty \left[\frac{|a_k|^2}{2} \cdot W_{N}(f_k) - 2 \norm{a}_{\ell^2(\C)}^2 \cdot \left((k+1) \cdot \eta_k + \sum_{j=k+1}^{\infty} \eta_j\right)\right]\\
      &= \sum_{k=n}^\infty \left(\frac{|a_k|^2}{2}\cdot W_{N}(f_k) - 2 \norm{a}_{\ell^2(\C)}^2 \cdot\varepsilon_k\right).
    \end{align*}
  
    Finally, let us derive \eqref{inequ: lower bound for W_N(f) (simplified)} from \eqref{inequ: lower bound for W_N(f) (technical)} under the additional assumptions (vi) and (vii). It suffices to show the statement under the assumption $a_k\neq 0$ for all $k \in \N_0$. The general statement then follows by throwing out all zero-terms from the input.
    Choosing $n=N$ in \eqref{inequ: lower bound for W_N(f) (technical)} and exploiting the upper bound for $\varepsilon_k$, $k \in \N$, we find that
    \begin{align}\label{inequ: (3)}
      W_N(f)\geq \sum_{k=N}^\infty \left(\frac{|a_k|^2}{2} \cdot W_{N}(f_k) - 2 \norm{a}_{\ell^2(\C)}^2 \cdot \varepsilon_k\right)
      \geq \sum_{k=N}^\infty \left(\frac{|a_k|^2}{2} \cdot W_{N}(f_k) - \delta \cdot |a_k|^2\right).
    \end{align}
    Since $W_N(f)$ is nonincreasing in $N$ for fixed $f \in L^2(\R^d)$, we have, for all $k \geq N$, 
    \[W_N(f_k)\geq W_k(f_k)\geq 4\delta.\] 
    Substituting this into \eqref{inequ: (3)} concludes the proof.
  \end{proof}
  
  We are now ready to state and prove a first theorem about arbitrarily slow energy propagation in scattering networks. This is a generalization of the idea outlined at the beginning of this section.
  
  \begin{theorem}\label{thm: Arbitrarily slow convergence of W_N}
    Let $F=(f_m)_{m\in \N_0} \subset L^2(\R^d)$. Suppose that for all tolerances $(\eta_k)_{k \in \N_0} \in \R_{>0}^{\N_0}$ there exist finite sets $\Psi_k\subseteq \Psi$, $k \in \N_0$, and a subsequence $(f_{m_k})_{k \in \N_0}$ of $F$ with $m_0=0$ such that the following hold:
    \begin{myenum}
      \item[(i)] The expression $C_F:=\sup_{m \in \N_0}\norm{f_m}_2^2$ is finite, i.e., $F$ is a bounded sequence.
      \item[(ii)] We have, for all $j,k \in \N_0$ with $j<k$, 
      \begin{align*}
        \left|\ip{f_{m_k},f_{m_j}}\right|\leq \eta_k.
      \end{align*}
      \item[(iii)] We have $\Psi_{k_1}\cap \Psi_{k_2}=\emptyset$ whenever $k_1\neq k_2$. 
      \item[(iv)] We have, for every $k \in \N_0$,
      \begin{align*}
        \sum_{\psi \in \Psi\setminus\Psi_k} \norm{f_{m_k}*\psi}_2^2 \leq \eta_k.
      \end{align*}  
      \item[(v)] We have, for every $k \in \N_0$ and $0\leq j < k$,
      \begin{align*}
        \sum_{\psi \in \Psi_k} \norm{f_{m_j}*\psi}_2^2 \leq \eta_k.
      \end{align*}
      \item[(vi)] There exists $\delta>0$ such that, for all $k \in \N$, $W_k(f_{m_k})\geq 4 \delta$.
    \end{myenum}
  
    If $E=(E_N)_{N \in \N}\in \R_{>0}^\N$ is a nonincreasing null-sequence, then there exists $f_E \in L^2(\R^d)$ that satisfies
    \begin{myenum}
      \item[a)] $\norm{f_E}_2^2\leq 2 C_F \cdot (1+E_1)$,
      \item[b)] $\norm{f_E-f_0}_2^2\leq 2 C_F \cdot E_1$, and
      \item[c)] for all $N\in \N$, $W_N(f_E)\geq \delta \cdot E_N$.
    \end{myenum}
  \end{theorem}
  
  \begin{proof}
    Define
    \begin{align*}
      (a_k)_{k \in \N}:=\left(\sqrt{E_k-E_{k+1}}\right)_{k \in \N},
    \end{align*} 
    and set $a_0:=1$. A telescoping argument gives, for every $N \in \N$,
    \begin{align*}
      \sum_{k=N}^\infty |a_k|^2 = \sum_{k=N}^\infty (E_k-E_{k+1}) = E_N.
    \end{align*}
    In particular, $a=(a_k)_{k \in \N_0} \in \ell^2(\N_0;\C)$ with $\norm{a}_{\ell^2(\N_0;\C)}^2=1+E_1$.
    
    We want to apply Lemma \ref{lem: approx additivity} to a subsequence $(f_{m_k})_{k \in \N_0}$ of $F$, which implies that $f_E$ can be chosen of the type 
    \[f_E:=\sum_{k=0}^\infty a_k f_{m_k}.\] 
    Here, $(f_{m_k})_{k \in \N_0}$ is the subsequence of $F$ from the prerequisites of this theorem that depends on tolerance levels $(\eta_k)_{k \in \N_0}$, which we now specify.
  
    Recall from Lemma \ref{lem: approx additivity} that $f_E$ converges in $L^2(\R^d)$ with an upper bound for its squared norm given by 
    \[\norm{f_E}_2^2\leq \sum_{k,j=0}^\infty\left|a_ka_j \cdot \ip{f_{m_k},f_{m_j}}\right|.\] 
    Likewise,   
    \[\norm{f_E-f_0}_2^2=\norm{\sum_{k=1}^\infty a_k f_{m_k}}_2^2 \leq \sum_{k,j=1}^\infty\left|a_ka_j \cdot \ip{f_{m_k},f_{m_j}}\right|.\] 
    We bound the latter series using the approximate orthogonality relation (ii). Repeatedly applying the Cauchy-Schwarz inequality yields, for all $n \in \N_0$,
    \begin{align*}\label{convergence of f}
      \sum_{k,j=n}^\infty&\left|a_ka_j \cdot \ip{f_{m_k},f_{m_j}}\right| \\
      & = \sum_{k=n}^\infty |a_k| \cdot \left(\sum_{j=n}^{k-1} |a_j| \cdot \underbrace{\left|\ip{f_{m_k},f_{m_j}}\right|}_{\leq \eta_k} 
      + |a_k| \cdot \underbrace{\left|\ip{f_{m_k},f_{m_k}}\right|}_{\leq C_F} 
      + \sum_{j=k+1}^\infty |a_j| \cdot \underbrace{\left|\ip{f_{m_k},f_{m_j}}\right|}_{\leq \eta_j}\right) \\
      & \leq C_F \cdot \sum_{k=n}^\infty |a_k|^2 
      +  \left(\sum_{j=n, j\neq k}^\infty |a_j|^2\right)^{\frac{1}{2}} \cdot \left(\sum_{k=n}^\infty |a_k| \cdot \left[(k-n)\cdot \eta_k^2 
      + \sum_{j=k+1}^\infty \eta_j^2\right]^{\frac{1}{2}}\right) \\
      & \leq C_F \cdot \sum_{k=n}^\infty |a_k|^2 
      +  \left(\sum_{j=n}^\infty |a_j|^2\right) \cdot \left(\sum_{k=n}^\infty \left[(k-n)\cdot \eta_k^2 
      + \sum_{j=k+1}^\infty \eta_j^2\right]\right)^{\frac{1}{2}}.       
    \end{align*}
    Thus, if we choose the tolerances $(\eta_k)_{k \in \N_0} \in \R_{>0}^{\N_0}$ small enough such that 
    \[\sum_{k=0}^\infty \left[k\cdot \eta_k^2 
    + \sum_{j=k+1}^\infty \eta_j^2\right]\leq C_F^2,\]
    then we obtain 
    \begin{align*}
      \sum_{k,j=n}^\infty\left|a_ka_j \cdot \ip{f_{m_k},f_{m_j}}\right|\leq 2C_F \cdot \left(\sum_{j=n}^\infty |a_j|^2\right) = \begin{cases}
        2C_F \cdot (1+E_1) &\text{ if } n=0\\
        2C_F \cdot E_1 &\text{ if } n=1
      \end{cases}.
    \end{align*}
    This proves the norm bounds in a) and b).
  
    As in Lemma \ref{lem: approx additivity}, define $\varepsilon_k:=(k+1)\cdot \eta_k+\sum_{j=k+1}^\infty \eta_j$ for all $k \in \N_0$. Since the tolerances $(\eta_k)_{k \in \N_0} \in \R_{>0}^{\N_0}$ are free to choose, we can additionally assume 
    that the series $\sum_{k=0}^\infty \varepsilon_k$ converges and that it holds, for all $k \in \N_0$,
    \[a_k=0 \qquad \text{or} \qquad \varepsilon_k\leq \frac{\delta \cdot |a_k|^2}{2\norm{a}_{\ell^2(\C)}^2}.\] By applying Lemma \ref{lem: approx additivity}, we conclude that, for all $N\in \N$,
    \[W_N(f_E)\geq \delta \cdot \sum_{k=N}^\infty |a_k|^2 = \delta \cdot E_N.\]
  \end{proof}
  
  Under certain assumptions on the high-pass filters $\Psi$, the latter theorem  guarantees that for any nonincreasing null-sequence $E=(E_N)_{N \in \N}\in \R_{>0}^\N$ there exists at least one input signal $f_E \in L^2(\R^d)$ whose network energy propagates slower across the network layers than $E$. Specifically, this implies that the set 
  \[\left\{f \in L^2(\R^d) ~\middle|~ W_N(f) \notin \mathcal{O}(E_N)\right\}\]
  is nonempty. Our next result provides a sufficient criterion under which this set is even dense in $L^2(\R^d)$.
  
  \begin{proposition}\label{prop: density of adversarial examples}
    Let $E=(E_N)_{N \in \N}\in \R_{>0}^{\N}$ be a nonincreasing null-sequence. Assume that for all $g \in L^2(\R^d)$ there are $ C_g,\delta_g>0$ such that the following holds: 
    For every nonincreasing null-sequence $E^\prime=(E_N^\prime)_{N \in \N}\in \R_{>0}^{\N}$ with $\liminf_{N \to \infty} \nicefrac{E_N}{E_N^\prime}=0$ there exists $f_{g,E^\prime} \in L^2(\R^d)$ such that
    \begin{myenum}
      \item[(i)] $\norm{f_{g,E^\prime}-g}_2^2\leq C_g \cdot E_1^\prime$, and
      \item[(ii)] $W_N(f_{g,E^\prime})\geq \delta_g \cdot E_N^\prime$ for all $N \in \N$.
    \end{myenum}
    Then, 
    \[Y_E:=\left\{f \in L^2(\R^d) ~\middle|~ W_N(f) \in \mathcal{O}(E_N)\right\}\]
    is a countable union of nowhere dense sets in $L^2(\R^d)$. In particular,  $L^2(\R^d) \setminus Y_E$ is dense in $L^2(\R^d)$.
  \end{proposition}
  
  \begin{proof} 
    Define, for all $p \in \N$,
    \[M_p:=\left\{f \in L^2(\R^d) ~\middle|~ \forall N \in \N:~W_N(f) \leq p \cdot E_N\right\},\] 
    and observe that $Y_E$ can be written as countable union of these sets,
    \[Y_E=\bigcup_{p \in \N} M_p.\] Our goal is to show that $M_p$ is nowhere dense in $L^2(\R^d)$ for every $p \in \N$. To this end, let us first recall from Corollary \ref{cor: W_N is continuous} that the nonlinear operator $W_N:L^2(\R^d) \to \R$ is continuous. Thus,
    \[M_p=\bigcap_{N \in \N} \left\{f \in L^2(\R^d) ~\middle|~ W_N(f) \leq p \cdot E_N\right\}\] is closed and it only remains to show that the interior of $M_p$ is empty: 
    Fix $g \in M_p$, $\varepsilon >0$, and define
    \begin{align*}
      E_N^\prime:=
      \begin{cases}
        \frac{\varepsilon^2}{2C_g} &\text{if } N=1\\
        E_1^\prime \cdot \sqrt{\frac{E_N}{E_1}} &\text{if } N \geq 2
      \end{cases}, \quad \text{for } N \in \N.
    \end{align*}
    Clearly, $E^\prime=(E_N^\prime)_{N \in \N}\in \R_{>0}^{\N}$ is a nonincreasing null-sequence that satisfies $\liminf_{N \to \infty} \nicefrac{E_N}{E_N^\prime}=0$. Hence, there is $f_{g,E^\prime}\in L^2(\R^d)$ that satisfies (i) and (ii) from the assumptions of this lemma.
    Thus, there is $N_0 \in \N$ such that 
    \[W_{N_0}(f_{g,E^\prime})\geq\delta_g \cdot E_{N_0}^\prime> p \cdot E_{N_0},\] which means \[f_{g,E^\prime}\in B_{\varepsilon}(g)\setminus M_p,\] 
    where $B_{\varepsilon}(g)$ denotes the open ball with radius $\varepsilon$ and center $g$ in $L^2(\R^d)$ with respect to $\norm{\cdot}_{2}$.
    Since this entails that the interior of $M_p$ is empty, we have shown in total that $Y_E$ is a countable union of nowhere dense sets in $L^2(\R^d)$. Therefore, its complement $L^2(\R^d) \setminus Y_E$ 
    is dense in $L^2(\R^d)$ by the Baire category theorem.
  \end{proof}
  
  For the remainder of this section, we examine the assumptions of Theorem \ref{thm: Arbitrarily slow convergence of W_N} and Proposition \ref{prop: density of adversarial examples}. While the number of conditions on the family $F$ and the high-pass filters $\Psi$ that must be met simultaneously for our results to apply is considerable, we demonstrate that these conditions can in fact be easily satisfied. To provide a concrete reference model, let us next establish that $F$ can be obtained by $L^2$-normalized dilations of a single function.
  
  \begin{lemma}\label{lem: ip of f and dilation of f converges to zero}
    Let $f \in L^2(\R^d)$. If $(A_m)_{m \in \N} \in \GL_d(\R)^\N$ 
    satisfies $\lim_{m \to \infty}\sigma_{\text{min}}(A_m)=\infty$, then
    \[\lim_{m \to \infty} \left|\ip{f,D_{A_m}^2f}\right|=0.\]
  \end{lemma}
  \begin{proof}
    The statement is obvious if $f=0$. For the other case, fix $f \in L^2(\R^d) \setminus \{0\}$ and $\varepsilon>0$. By Lebesgue's dominated convergence theorem there is $R>0$ so that $f_R:=f \cdot \mathds{1}_{B_R(0)}$ satisfies
    \[\norm{f-f_R}_2 \leq \frac{\varepsilon}{2\norm{f}_2}.\]
    For $m \in \N$, splitting the inner product gives
    \begin{align}\label{splitting the inner product}
      \left|\ip{f,D_{A_m}^2f}\right|
       \leq \left|\ip{f,D_{A_m}^2f_R}\right|+\left|\ip{f,D_{A_m}^2(f-f_R)}\right|
    \end{align} 
    We obtain an upper bound for the second term by the Cauchy-Schwarz inequality, 
    \begin{align}\label{estimate of second term}
      \left|\ip{f,D_{A_m}^2(f-f_R)}\right| \leq \norm{f}_2 \cdot \norm{D_{A_m}^2(f-f_R)}_2 = \norm{f}_2 \cdot \norm{f-f_R}_2 \leq \frac{\varepsilon}{2}.
    \end{align}
    We now turn our attention to the first term in \eqref{splitting the inner product}. As can be seen from a singular value decomposition of $A_m$, we have, for all $x \in \R^d$,
    \begin{align*}
      \norm{A_mx}_2 \geq \sigma_{\text{min}}(A_m) \cdot \norm{x}_2.
    \end{align*}
    Since by assumption $\lim_{m \to \infty}\sigma_{\text{min}}(A_m)=\infty$, we find that, for all $x \in \R^d \setminus \{0\}$,
    \begin{align*}
      \mathds{1}_{B_R(0)}(A_mx) \leq \mathds{1}_{B_\frac{R}{\sigma_{\text{min}}(A_m)}(0)}(x) \xrightarrow{m \to \infty} 0. 
    \end{align*}
    Hence, by Lebesgue's dominated convergence theorem there is $m_0 \in \N$ so that, for all $m \geq m_0$,
    \begin{align*}
      \norm{f \cdot \mathds{1}_{B_\frac{R}{\sigma_{\text{min}}(A_m)}(0)}}_2 \leq \frac{\varepsilon}{2 \norm{f}_2}.
    \end{align*} 
    Thus, for all $m \geq m_0$,
    \begin{align}\label{estimate of first term}
      \left|\ip{f,D_{A_m}^2f_R}\right| & \leq
      \int_{\R^d} |f(x)| \cdot |D_{A_m}^2f(x)| \cdot \mathds{1}_{B_R(0)}(A_mx) \dx \nonumber \\
      & \leq \int_{\R^d} |f(x)| \cdot |D_{A_m}^2f(x)| \cdot \mathds{1}_{B_\frac{R}{\sigma_{\text{min}}(A_m)}(0)}(x) \dx \nonumber        \\
      & \leq \norm{f \cdot \mathds{1}_{B_\frac{R}{\sigma_{\text{min}}(A_m)}(0)}}_2 \cdot \norm{D_{A_m}^2f}_2 \leq \frac{\varepsilon}{2}.
    \end{align}
    Plugging \eqref{estimate of second term} and \eqref{estimate of first term} into \eqref{splitting the inner product} concludes the proof.
  \end{proof}
  
  The following lemma shows that $F$ satisfies assumptions (i)--(v) from Theorem \ref{thm: Arbitrarily slow convergence of W_N} if it is generated by $L^2$-normalized matrix-dilations of a fixed function for matrices that asymptotically become increasingly expansive.
  
  \begin{lemma}\label{lem: Matrix dilations satisfy assumptions for approx add}
    Let $(A_m)_{m \in \N} \in \GL_d(\R)^\N$ be so that $\lim_{m \to \infty}\sigma_{\text{min}}(A_m)=\infty$. Let $f \in L^2(\R^d)$, and define $f_m:=D_{A_m}^2f$, $m \in \N$. Finally, choose $f_0 \in \{0,f\}$.
    
    Then, $F=(f_m)_{m \in \N_0}$ satisfies assumptions (i)--(v) from Theorem \ref{thm: Arbitrarily slow convergence of W_N}.
  \end{lemma}
  
  \begin{proof}
    The statement is trivial if $f=0$. Hence, suppose $f \in L^2(\R^d)\setminus\{0\}$ in the following. Furthermore, assumptions (i)-(v) of Theorem \ref{thm: Arbitrarily slow convergence of W_N} are clearly easier to meet if we are in the case that $f_0=0$ (instead of $f_0=f$). Therefore, we only spell out the proof for $f_0=f$. For notational convenience, we set $A_0:=I_d$, so we can write $f_0=D_{A_0}^2f$.
  
    By the normalization of the dilations, $C_F:=\sup_{m\in \N_0}\norm{f_m}_2^2=\norm{f}_2^2$ is finite, i.e., assumption (i) of Theorem \ref{thm: Arbitrarily slow convergence of W_N} is satisfied. We now have to show that for all tolerances $(\eta_k)_{k \in \N_0} \in \R_{>0}^{\N_0}$ there exist finite sets $\Psi_k\subseteq \Psi$, $k \in \N_0$, and a subsequence $(f_{m_k})_{k \in \N_0}$ of $F$ with $m_0=0$ such that assumptions (ii)-(v) of Theorem \ref{thm: Arbitrarily slow convergence of W_N} hold.
    We proceed by choosing the sets $\Psi_k\subseteq \Psi$, $k \in \N_0$, and $({m_k})_{k \in \N}\in \N^\N$ inductively.
    
    Starting with $k=0$ and $m_0=0$, we only need to show that condition (iv) can be satisfied. Since 
    \[\sum_{\psi \in \Psi} \norm{f_{m_0}*\psi}_2^2\leq \norm{f}_2^2<\infty,\] there is a finite subset $\Psi_0\subset \Psi$ such that 
    \begin{align*}
      \sum_{\psi \in \Psi\setminus\Psi_0} \norm{f_{m_0}*\psi}_2^2 \leq \eta_0.
    \end{align*} 
    For the induction step, let $k \in \N$ and assume that $\Psi_0,\ldots,\Psi_{k-1}$ and $m_0,\ldots,m_{k-1}$ are such that (ii)-(v) hold. Our goal is now to obtain $\Psi_k$ and $m_k$ with the desired properties. 
    
    It is easy to meet condition (ii): For all $m \in \N_0$ and $0\leq j <k$, we have
    \begin{align*}
      \left|\ip{f_m,f_{m_j}}\right|=\left|\ip{D_{A_{m_j}^{-1}}^2D_{A_{m}}^2f,f}\right|
      =\left|\ip{D_{A_{m}A_{m_j}^{-1}}^2f,f}\right|.
    \end{align*}
    Moreover,
    \[\sigma_{\text{min}}\left(A_{m}A_{m_j}^{-1}\right)
    \geq \sigma_{\text{min}}\left(A_m\right) \cdot \sigma_{\text{min}}\left(A_{m_j}^{-1}\right)
    =\frac{\sigma_{\text{min}}\left(A_{m}\right)}{\sigma_{\text{max}}\left(A_{m_j}\right)}
    \geq \frac{\sigma_{\text{min}}\left(A_{m}\right)}{\max_{0\leq l<k}\sigma_{\text{max}}\left(A_{m_l}\right)}.\]
    Since the latter bound is uniformly in $j$ for $0\leq j<k$, and $\lim_{m \to \infty}\sigma_{\text{min}}\left(A_{m}\right)=\infty$ by assumption, Lemma \ref{lem: ip of f and dilation of f converges to zero} implies that there exists $m_k^\prime \in \N_0$ such that $\left|\ip{f_{m_k},f_{m_j}}\right|<\eta_k$ holds for all $j \in \{0,\ldots,k-1\}$ whenever we choose $m_k$ larger than $m_k^\prime$. 
  
    The harder part is to satisfy the remaining conditions (iii)-(v) simultaneously. 
    We begin with a simple observation about the impact of matrix-dilations on the frequency content of $f$.
    For $R>r>0$, let 
    \[P_{r,R}f:=\mathcal{F}^{-1}(\mathcal{F}(f)\cdot \mathds{1}_{S_{r,R}})\] 
    be the projection of $f$ on the subspace of $L^2(\R^d)$ whose Fourier transform is supported in $S_{r,R}$. We introduce $\varepsilon \in (0,\nicefrac{1}{2})$, which we will choose later in the proof (small enough). By Lebesgue's dominated convergence theorem, $\widehat{f}$ is concentrated in a spherical shell $S_{r(\varepsilon),R(\varepsilon)}$ up to $\varepsilon$-error for some $R(\varepsilon)>r(\varepsilon)>0$, meaning that
    \[\norm{P_{r(\varepsilon),R(\varepsilon)}f}_2^2
    \geq (1-\varepsilon) \cdot \norm{f}_2^2.\] 
    Now, if $A \in \GL_d(\R)$, then $\widehat{D_A^2f}$ is concentrated in $S_{\sigma_{\text{min}}(A)\cdot r(\varepsilon),\sigma_{\text{max}}(A) \cdot R(\varepsilon)}$ up to $\varepsilon$-error, since
    \begin{align*}
      \norm{P_{\sigma_{\text{min}}(A)\cdot r(\varepsilon),\sigma_{\text{max}}(A) \cdot R(\varepsilon)}D_A^2f}_2^2 
      &= \norm{\widehat{D_A^2f}\cdot \mathds{1}_{S_{\sigma_\text{min}(A)\cdot r(\varepsilon),\sigma_{\text{max}}(A)\cdot R(\varepsilon)}}}_2^2 \\
      &=\int_{\R^d}|\det(A)|^{-1} \cdot |\widehat{f}(A^{-T}\xi)|^2 \cdot \mathds{1}_{S_{\sigma_\text{min}(A)\cdot r(\varepsilon),\sigma_{\text{max}}(A)\cdot R(\varepsilon)}}(\xi) \dxi \\
      &=\int_{\R^d}|\widehat{f}(z)|^2 \cdot \mathds{1}_{S_{\sigma_\text{min}(A)\cdot r(\varepsilon),\sigma_{\text{max}}(A)\cdot R(\varepsilon)}}(A^Tz) \dz \\
      &\geq\int_{\R^d}|\widehat{f}(z)|^2 \cdot \mathds{1}_{S_{r(\varepsilon),R(\varepsilon)}}(z) \dz \\
      &=\norm{\mathcal{F}P_{r(\varepsilon),R(\varepsilon)}f}_2^2 \geq (1-\varepsilon) \cdot \norm{f}_2^2.
    \end{align*}
    Specifically, setting 
    \begin{align*}
      r_{<k}:=\min_{0\leq j<k} \sigma_{\text{min}}(A_{m_j}) \cdot r(\varepsilon) \quad \text{and} \quad R_{<k}:=\max_{0\leq j<k} \sigma_{\text{max}}(A_{m_j}) \cdot R(\varepsilon),
    \end{align*}
    the above frequency concentration estimate entails
    \begin{align*}
      \norm{(\id_{L^2(\R^d)}-P_{r_{<k},R_{<k}})f_{m_j}}_2^2 \leq \varepsilon \cdot \norm{f}_2^2 \quad \text{for } 0\leq j <k.
    \end{align*}
    By the Littlewood-Paley condition \eqref{ass: Littlewood-Paley condition} and the continuity of the Fourier transform of the filters, there is a finite subset $\widetilde{\Psi}\subseteq \Psi$ containing $\bigcup_{j=0}^{k-1}\Psi_j$, and such that
    \begin{align}\label{inequ: (5)}
      |\widehat{\chi}(\xi)|^2+\sum_{\psi \in \widetilde{\Psi}}|\widehat{\psi}(\xi)|^2 > 1-\varepsilon \quad \text{ for all } \xi \in S_{r_{<k},R_{<k}}.
    \end{align} 
    By the Riemann-Lebesgue lemma, there exists $T\geq R_{<k}$ such that
    \begin{align}\label{inequ: (4)}
      |\widehat{\chi}(\xi)|^2+\sum_{\psi \in \widetilde{\Psi}}|\widehat{\psi}(\xi)|^2 < \varepsilon \quad \text{ whenever } |\xi| \geq T.
    \end{align}
    Since $\lim_{m \to \infty}\sigma_{\text{min}}\left(A_{m}\right)=\infty$ by assumption, there is $m_k^{\prime\prime}\in \N$ with $\sigma_{\text{min}}(A_{m}) \cdot r(\varepsilon)\geq T$ for all $m \geq m_k^{\prime\prime}$.
    Define $m_k:=\max\{m_k^\prime,m_k^{\prime\prime}\}$, $r_k:=\sigma_{\text{min}}(A_{m_k}) \cdot r(\varepsilon)$, and $R_k:=\sigma_{\text{max}}(A_{m_k}) \cdot R(\varepsilon)$. Thus,
    \begin{align*}
      \norm{(\id_{L^2(\R^d)}-P_{r_{k},R_{k}})f_{m_k}}_2^2 \leq \varepsilon \cdot \norm{f}_2^2.
    \end{align*}
    Moreover, because of \eqref{inequ: (4)}, the Littlewood-Paley condition \eqref{ass: Littlewood-Paley condition}, and the continuity of the Fourier transform of the filters, there exists a finite subset $\Psi(k,\varepsilon)\subseteq \Psi \setminus \widetilde{\Psi}$ so that
    \begin{align*}
      \sum_{\psi \in \Psi(k,\varepsilon)}|\widehat{\psi}(\xi)|^2 > 1-\varepsilon \quad \text{ for all } \xi \in S_{r_k,R_k},
    \end{align*}
    which itself entails
    \begin{align}\label{inequ: (7)}
      \sum_{\psi \in \Psi \setminus \Psi(k,\varepsilon)}|\widehat{\psi}(\xi)|^2 < \varepsilon \quad \text{ for all } \xi \in S_{r_k,R_k}.
    \end{align}
    At the same time, \eqref{inequ: (5)} and the fact that $\Psi(k,\varepsilon)\subseteq \Psi \setminus \widetilde{\Psi}$, imply
    \begin{align}\label{inequ: (6)}
      \sum_{\psi \in \Psi(k,\varepsilon)}|\widehat{\psi}(\xi)|^2 < \varepsilon \quad \text{ for all } \xi \in S_{r_{<k},R_{<k}}.
    \end{align}
    We claim that $\Psi_k:=\Psi(k,\varepsilon)$ with $\varepsilon:=\frac{\eta_k}{2\norm{f}_2^2}$ does the job. 
  
    In fact, first note that conditions (i)-(iii) are immediately satisfied by the above construction. Concerning condition (iv), we have 
      \begin{align*}
        \sum_{\psi \in \Psi\setminus\Psi_k} \norm{f_{m_k}*\psi}_2^2 
        &= \sum_{\psi \in \Psi\setminus\Psi_k} \norm{(P_{r_k,R_k}f_{m_k})*\psi}_2^2 + \sum_{\psi \in \Psi\setminus\Psi_k} \norm{((\id_{L^2(\R^d)}-P_{r_k,R_k})f_{m_k})*\psi}_2^2 \\
        &\leq \int_{\R^d} |\widehat{f_{m_k}}(\xi)|^2 \cdot \sum_{\psi \in \Psi\setminus\Psi_k} |\widehat{\psi}(\xi)|^2 \cdot \mathds{1}_{S_{r_k,R_k}}(\xi)\dxi + \norm{(\id_{L^2(\R^d)}-P_{r_k,R_k})f_{m_k}}_2^2 \\
        &\leq \varepsilon \cdot \int_{\R^d} |\widehat{f_{m_k}}(\xi)|^2 \dxi + \varepsilon \cdot \norm{f}_2^2 \\
        &= \varepsilon \cdot \norm{f_{m_k}}_2^2+\varepsilon \cdot \norm{f}_2^2 = \eta_k.
      \end{align*}  
      Finally, we obtain condition (v) by proceeding analogously to the previous step. Indeed, for $0\leq j<k$, we have
      \begin{align*}
        \sum_{\psi \in \Psi_k} \norm{f_{m_j}*\psi}_2^2 
        &= \sum_{\psi \in \Psi_k} \norm{(P_{r_{<k},R_{<k}}f_{m_j})*\psi}_2^2 + \sum_{\psi \in \Psi_k} \norm{((\id_{L^2(\R^d)}-P_{r_{<k},R_{<k}})f_{m_j})*\psi}_2^2 \\
        &\leq \int_{\R^d} |\widehat{f_{m_j}}(\xi)|^2 \cdot \sum_{\psi \in \Psi_k} |\widehat{\psi}(\xi)|^2 \cdot \mathds{1}_{S_{r_{<k},R_{<k}}}(\xi)\dxi + \norm{(\id_{L^2(\R^d)}-P_{r_{<k},R_{<k}})f_{m_j}}_2^2 \\
        &\leq \varepsilon \cdot \int_{\R^d} |\widehat{f_{m_j}}(\xi)|^2 \dxi + \varepsilon \cdot \norm{f}_2^2 \\
        &= \varepsilon \cdot \norm{f_{m_j}}_2^2+\varepsilon \cdot \norm{f}_2^2 = \eta_k.
      \end{align*}
  \end{proof}
  
  We summarize our previous results in the case where $F$ is generated by $L^2$-normalized matrix-dilations.
  
  \begin{corollary}\label{cor: arbitrarily slow decay if delta is positive}
    Let $(A_m)_{m \in \N} \in \GL_d(\R)^\N$ be so that $\lim_{m \to \infty}\sigma_{\text{min}}(A_m)=\infty$. Suppose that
    \[\delta:=\inf_{\substack{f \in L^2(\R^d),\\ \norm{f}_2=1}}\inf_{k \in \N} \limsup_{m \to \infty} W_k(D_{A_m}^2f)>0.\] 
    Then, for every $g \in L^2(\R^d)$ there exists a universal constant $C_g>0$ with the following property:
    
    For any nonincreasing null-sequence $E=(E_N)_{N \in \N}\in \R_{>0}^\N$ there is $f_{g,E} \in L^2(\R^d)$ such that
    \begin{myenum}
      \item[a)] $\norm{f_{g,E}}_2^2\leq C_g \cdot (1+E_1)$,
      \item[b)] $\norm{f_{g,E}-g}_2^2\leq C_g \cdot E_1$, and
      \item[c)] $W_N(f_{g,E})\geq \frac{\norm{g}_2^2}{8} \cdot \delta \cdot E_N$ for all $N\in \N$.
    \end{myenum}
    Furthermore,
    \[Y_E:=\left\{f \in L^2(\R^d) ~\middle|~ W_N(f) \in \mathcal{O}(E_N)\right\}\]
    is a countable union of nowhere dense sets in $L^2(\R^d)$. In particular,  $L^2(\R^d) \setminus Y_E$ is dense in $L^2(\R^d)$.
  \end{corollary}
  
  \begin{proof}
    Let us first assume that $g\neq 0$. Define $g_m:=D_{A_m}^2g$ for $m \in \N_0$, where $A_0:=I_d$. From the assumption that $\delta>0$, we conclude that there exists a strictly increasing sequence $(m_k)_{k\in \N_0} \subseteq \N_0$ with $m_0=0$, and such that $W_k\left(\frac{g_{m_k}}{\norm{g}_2}\right)\geq \frac{1}{2} \cdot \delta$, or equivalently $W_k\left(g_{m_k}\right)\geq \frac{\norm{g}_2^2}{2} \cdot \delta$ for all $k \in \N$. The existence of $f_{g,E}$ that satisfies the properties a)--c) then follows from Theorem \ref{thm: Arbitrarily slow convergence of W_N} and Lemma \ref{lem: Matrix dilations satisfy assumptions for approx add}.
  
    For $g=0$, fix any $h \in L^2(\R^d)$ with $\norm{h}_2=1$. Set $h_0:=0\in L^2(\R^d)$ and $h_m:=D_{A_m}^2h$ for $m \in \N$. Again from the assumption that $\delta>0$, we conclude that there is a strictly increasing sequence $(m_k)_{k\in \N_0} \subseteq \N_0$ with $m_0=0$, and such that, for all $k \in \N$, $W_k\left(h_{m_k}\right)\geq \frac{1}{2} \cdot \delta$. Applying Theorem \ref{thm: Arbitrarily slow convergence of W_N} and Lemma \ref{lem: Matrix dilations satisfy assumptions for approx add} to the sequence $(h_{m_k})_{k \in \N}$ then yields $f_{g,E}:=f_{h,E}\in L^2(\R^d)$ that satisfies $\norm{f_{g,E}}_2^2\leq 2E_1$ and, for all $N \in \N$, $W_N(f_{g,E})\geq \frac{1}{8} \cdot \delta \cdot E_N$.
  
    Overall, this also shows that we can apply Proposition \ref{prop: density of adversarial examples}, which completes the proof. 
  \end{proof}
  
  We are now ready to present the main theorem of this section. It states that the prerequisites of Corollary \ref{cor: arbitrarily slow decay if delta is positive} are satisfied if the high-pass filters $\Psi$ admit an inclusive structure when $L^1$-dilated by the matrix sequence $(A_m^{-1})_{m \in \N}$. Specifically, this includes filter banks, where the high-pass filters $\Psi$ arise from the $L^1$-normalized dilations by the matrix sequence $(A_m)_{m \in \N}$ applied to a finite number of generators, as is the case for wavelet-generated filter banks. In these cases, both the result of arbitrarily slow energy propagation and its consequence concerning the density of signals that fail to meet a given decay rate apply.
  
  \begin{theorem}\label{thm: slow propagation if filters are dilation invariant}
    Let $(A_m)_{m \in \N} \in \GL_d(\R)^\N$ be such that the following are true:
    \begin{myenum}
      \item[(i)] We have  
      $\lim_{m \to \infty}\sigma_{\text{min}}(A_m)=\infty$.
      \item[(ii)] The inclusion $\Psi_{m} \subseteq \Psi_{m+1}$ holds for all $m \in \N$, where $\Psi_m:=D_{A_m^{-1}}^1\Psi$.
    \end{myenum}
  
    Then, $\Psi_\infty:=\bigcup_{m \in \N} \Psi_m$ is a semi-discrete Parseval frame. Moreover, for all $f \in L^2(\R^d)$, $k \in \N$,
    \begin{align*}
      \lim_{m \to \infty} W_k[\Psi]\left(D_{A_m}^2f\right)=\norm{f}_2^2.
    \end{align*}
    In particular, Corollary \ref{cor: arbitrarily slow decay if delta is positive} applies with $\delta=1$.
  \end{theorem}
  
  \begin{proof}
    By the Littlewood-Paley condition \eqref{ass: Littlewood-Paley condition}, we have, for every $m \in \N$, 
    \begin{align*}
      |\hat{\chi}(A_m^{T}\xi)|^2+\sum_{\psi \in \Psi_m} |\hat{\psi}(\xi)|^2 = |\hat{\chi}(A_m^{T}\xi)|^2+\sum_{\psi \in \Psi} |\hat{\psi}(A_m^{T}\xi)|^2=1 \quad \text{a.e. } \xi \in \R^d.
    \end{align*}
    In particular, for all $f \in L^2(\R^d)$, $m \in \N$, 
    \[\sum_{\psi \in \Psi_m} \norm{f*\psi}_2^2  \leq \norm{f}_2^2,\]  
    which implies 
    \[\sum_{\psi \in \Psi_\infty} \norm{f*\psi}_2^2\leq \norm{f}_2^2.\] 
    On the other hand, as a consequence of $\lim_{m \to \infty}\sigma_{\text{min}}(A_m^{T})=\infty$, the Riemann-Lebesgue lemma, and Lebesgue's dominated convergence theorem, we find that
    \[\sum_{\psi \in \Psi_\infty} \norm{f*\psi}_2^2 \geq \sum_{\psi \in \Psi_m} \norm{f*\psi}_2^2 = \norm{f}_2^2 - \norm{f*D_{A_m^{-1}}^1\chi}_2^2 \xrightarrow{m \to \infty} \norm{f}_2^2.\] 
    Thus, $\Psi_\infty$ is a semi-discrete Parseval frame.
  
    Fix $f \in L^2(\R^d)$. Recall from Lemma \ref{lemma: commuting U[p] with D_A} that we have, for all $k \in \N$, $m \in \N$, 
    \begin{align*}
      W_k[\Psi](D_{A_m}^2f)=W_k[D_{A_m^{-1}}^1\Psi](f)=W_k[\Psi_m](f).
    \end{align*}
    Let $\varepsilon>0$. Since $W_k[\Psi_\infty](f)=\norm{f}_2^2$ by Lemma \ref{lem: U is norm-preserving if we consider the whole Parseval frame}, there is a finite subset $\tilde{\Psi}=\tilde{\Psi}(\varepsilon)\subset \Psi_\infty$ such that 
    \[W_k[\widetilde{\Psi}](f)\geq \norm{f}_2^2-\varepsilon.\] 
    Since $\widetilde{\Psi}$ is finite, there exists $M=M(\varepsilon) \in \N$ with $\widetilde{\Psi} \subset \Psi_{M}$. Thus, for all $m \geq M$,
    \[\norm{f}_2^2\geq W_k[\Psi_{m}](f)\geq W_k[\Psi_{M}](f) \geq W_k[\widetilde{\Psi}](f) \geq \norm{f}_2^2-\varepsilon,\]
    which concludes the second part of the theorem.
  \end{proof}
  
  In our main application - concerning the wavelet scattering transform; see Section \ref{sec: Scattering with wavelets} - of the previous theorem, the matrices $(A_m)_{m \in \N}$ are generated by a single expansive matrix immediately implying that condition (i) from Theorem \ref{thm: slow propagation if filters are dilation invariant} is fulfilled.
  We conclude this section by explicitly \mbox{(re-)stating} our main negative findings for this particular instance of our assumptions, which is much easier to verify than the assumptions of Theorem \ref{thm: Arbitrarily slow convergence of W_N}.
  
  \begin{corollary}\label{cor: slow propagation in case of a single expansive matrix}
    Let $A \in \GL_d(\R)$ be such that $\sigma_{\min}(A)>1$. If $\Psi_{m} \subseteq \Psi_{m+1}$ holds for all $m \in \N$, where $\Psi_m:=D_{A^{-m}}^1\Psi$, then $\Psi_\infty:=\bigcup_{m \in \N} \Psi_m$ is a semi-discrete Parseval frame. Moreover, for all $f \in L^2(\R^d)$, $k \in \N$,
    \begin{align*}
      \lim_{m \to \infty} W_k[\Psi]\left(D_{A^m}^2f\right)=\norm{f}_2^2.
    \end{align*}
    In particular, Corollary \ref{cor: arbitrarily slow decay if delta is positive} applies with $\delta=1$.
  \end{corollary}
  
  \section{Convergence Rates for Scattering Propagation}\label{sec: Convergence rates for scattering propagation}
  
  In this section, we show that under mild analyticity assumptions on the scattering filters $\Psi$, fast energy propagation can still be achieved for large signal classes in $L^2(\R^d)$. These signal classes depend on the required speed of energy propagation, i.e., the faster the energy propagation, the smaller the signal class for which we can \textit{guarantee} this speed. This is reflected by an interplay between the decay behavior of the Fourier transform of the input signal relative to the filters and the frequency localization of the filters regarding the size and shape of their Fourier supports.
    
  Notably, our results in this section are general, since
  \begin{itemize}
    \item they hold in any dimension $d \in \N$;
    \item our assumptions can be easily satisfied; and
    \item the setup provides several parameters that can be used for fine-tuning.
  \end{itemize}
  In particular, Section \ref{sec: Scattering with wavelets} demonstrates that these results apply to filter banks generated by bandlimited wavelets, thereby complementing the negative results from the previous section about arbitrarily slow energy propagation in such scattering networks.
   
  Our analysis begins with a similar approach to that used in \cite[Theorem 3.1]{waldspurger2017exponential} and \cite[Theorem 1, Theorem 2]{wiatowski2017energy}, establishing an upper bound on $W_N[\Psi](f)$, for all $f \in L^2(\R^d)$ and $N \in \N_{\geq 2}$, which is of the type
  \begin{align}\label{eq: bound on W_N(f) by integration against K_N}
    W_N[\Psi](f) \leq \int_{\R^d} |\widehat{f}(\xi)|^2 \cdot K_N(\xi) \dxi
  \end{align}
  for a family $(K_N)_{N \in \N_{\geq 2}}$ of integral kernels (independent of $f$) that satisfy additional properties. In parts, we combine the techniques from \cite{waldspurger2017exponential,wiatowski2017energy} in the proof of Theorem \ref{thm: W_N(f) is upper-bounded by integral}, which allows for more flexibility in our assumptions on the filters. On the basis of this approach, we then leverage the inductive structure of the scattering network to involve the filters $\Psi$ in the integral in \eqref{eq: bound on W_N(f) by integration against K_N}.
  
  To clarify our assumptions for this section, we introduce some additional notation. Let $\nu^* \in \mathbb{S}^{d-1}$ be the vector with entries $\nu_k^*:=\nicefrac{1}{\sqrt{d}}$ for each $k \in \{1,\ldots,d\}$. For $\rho \in [0,1)$, we  denote by 
  \begin{align}\label{def: cone}
    C^{\rho}:=\left\{x \in \R_{\geq 0}^d ~\middle|~ \ip{x,\nu^*} \geq (1-\rho) \cdot \norm{x}_2 \right\}
  \end{align}
  the closed cone in the first canonical orthant with its tip at the origin and opening angle relative to $\nu^*$ parameterized by $\rho$. Note that $C^{\rho}$ is precisely the first canonical orthant $\R_{\geq 0}^d$ if $\rho \geq 1-\frac{1}{\sqrt{d}}$. In particular, in dimension $d=1$, we have, for all $\rho \in [0,1)$, $C^\rho=[0,\infty)$. The following assumptions are motivated by \cite[Assumption 1]{wiatowski2017energy}.
  
  \begin{assumption}\label{ass: high-pass nature and analyticity of filters}
  Let $\kappa \in \N_{\geq 2}$. Let $(r_j)_{j \in \N} \in \R_{>0}^{\N}$ be a strictly increasing sequence of scales and suppose that $\gamma:=\gamma(\kappa):=\inf_{j \in \N} \frac{r_{j}}{r_{j+\kappa}}>0$. In addition to the Littlewood-Paley condition \eqref{ass: Littlewood-Paley condition}, we assume that there exists a partition $\Psi = \bigcup_{j \in \N} \Psi_j$ such that for every $\psi \in \Psi_j$, $j \in \N$, there is an orthogonal matrix $A_{\psi} \in O_d(\R)$ such that
  \begin{align*}
  \supp(\widehat{\psi})\subseteq C_{A_{\psi}}^\rho \cap \s_{r_{j},r_{j+\kappa}},
  \end{align*}
  where $C_{A_{\psi}}^\rho:=\{A_{\psi} x ~|~ x \in C^\rho\}$.
  \end{assumption}

  \begin{figure}[h]
    \centering
    \begin{minipage}[h]{0.45\textwidth}
        \centering
        \includegraphics[width=0.95\textwidth]{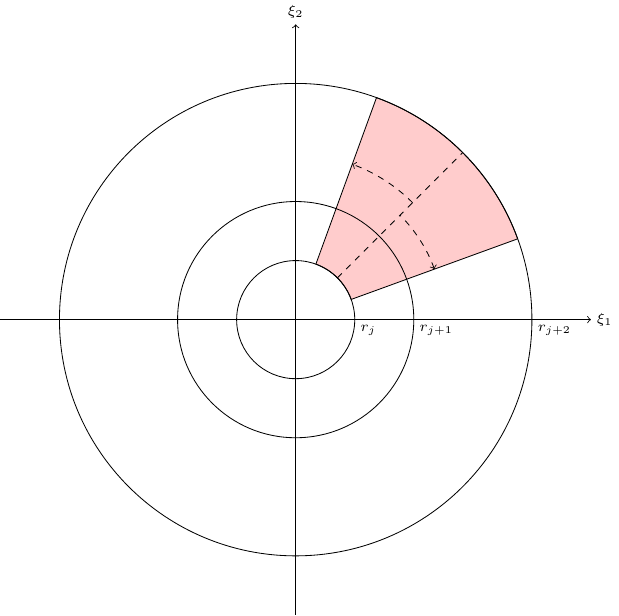}
        \label{fig:sub1}
    \end{minipage}
    \hfill
    \begin{minipage}[h]{0.45\textwidth}
        \centering
        \includegraphics[width=0.95\textwidth]{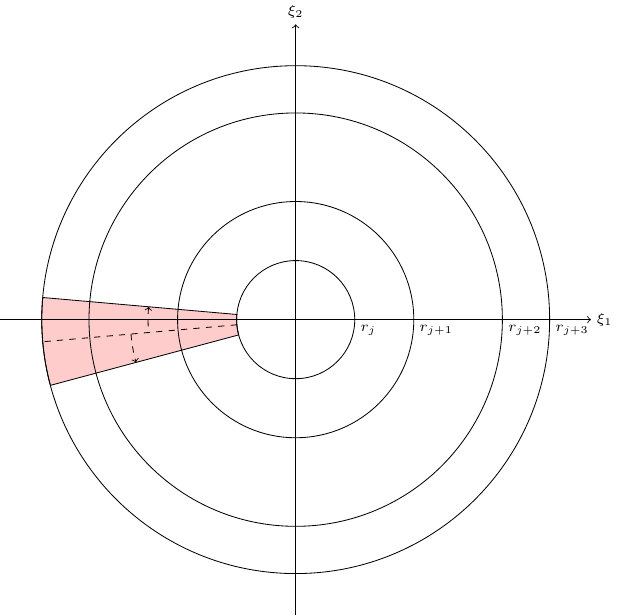}
        \label{fig:sub2}
    \end{minipage}
    \caption{Illustration of the allowed regions for the frequency supports of two different filters from (possibly) two different families of high-pass filters. 
    The parameter $\kappa$ (left $\kappa=2$, right $\kappa=3$) reflects the number of scales the filter may interfere with, and $\rho$ parameterizes the opening angle of the colored segment.}
    \label{fig:both}
  \end{figure}
  
  The present setting may be regarded as a refinement of \cite[Assumption 1]{wiatowski2017energy}, in that it incorporates additional information on the frequency localization of the filters, which ultimately affects the decay behavior of the energy remainder. We think of this as follows: By the Littlewood-Paley condition \eqref{ass: Littlewood-Paley condition}, the filters $\Psi$ induce a decomposition of the frequency space $\R^d$ according to 
  \begin{align}\label{eq:_decomposition_of_frequency_space}
  	\R^d=B_{r_1}(0) \cup \bigcup_{j\in \N} \bigcup_{\psi \in \Psi_j} (C_{A_\psi}^\rho \cap \s_{r_{j},r_{j+\kappa}}).
  \end{align}
  The filters $\Psi$ are of high-pass nature in the sense that they leave a frequency gap around zero, 
  \begin{align}\label{frequency gap}
    \sum_{\psi \in \Psi} |\widehat{\psi}(\xi)|^2=0 \quad \text{ for all } \xi \in B_{r_1}(0).
  \end{align}
  By the Littlewood-Paley condition \eqref{ass: Littlewood-Paley condition}, this gap is filled by $\widehat{\chi}$. Moreover, any high-pass filter is frequency-localized in a region that is the intersection of a spherical shell, and a (rotated) cone with tip at the origin and angular spread parameterized by $\rho$. The maximum number of scale-interferences is specified by $\kappa$. The parameter $\gamma$ reflects the (inverse of the) maximal relative spread between the inner and outer radius of any such shell. Clearly, we have $\gamma \in (0,1)$, since the scales are strictly increasing. The closer $\gamma$ is to $1$ and the closer $\rho$ is to $0$, the more localized the high-pass filters are in frequency space (Fig.~\ref{fig:both}).
  
  These additional parameters are reflected in the decay rates that we can guarantee for the energy remainder of certain signals, offering new insight into the interplay between the localization of the filters (in terms of size and shape of their Fourier supports) and the frequency decay behavior of the signals. We note however that our assumptions implicitly require the filters to be bandlimited, which is slightly more restrictive than the setting considered in \cite[Assumption 1]{wiatowski2017energy}.

  We begin our analysis with two auxiliary results about the geometry of the allowed domain for the Fourier supports of the filters. The first one provides an upper bound on the maximum distance from $\nu^*$ to any other point in the set $C^\rho \cap \mathbb{S}^{d-1}$. 
  \begin{lemma}\label{lemma: maximum distance to v*}
    If $t>0$, then $\max_{x \in C^\rho \cap \mathbb{S}^{d-1}} \norm{x-t\nu^*}_2^2\leq t^2-2t \cdot (1-\rho)+1.$
  \end{lemma}
  
  \begin{proof}
    A direct computation shows that, for all $x \in C^\rho \cap \mathbb{S}^{d-1}$,
    \begin{align*}
      \norm{x-t\nu^*}_2^2 = \norm{x}_2^2 - 2t \cdot \ip{x,\nu^*} + t^2 \cdot \norm{\nu^*}_2^2 \leq t^2-2t \cdot (1-\rho) \cdot \norm{x}_2+1 = t^2-2t \cdot (1-\rho)+1.
    \end{align*}
  \end{proof}
  
  The next lemma is crucial for the induction step in the proof of Theorem \ref{thm: W_N(f) is upper-bounded by integral}. 
  
  \begin{lemma}\label{lem: how alpha comes into play}
    There is a family $(\nu_\psi)_{\psi \in \Psi}\subseteq \R^d$ so that, for all $\psi \in \Psi$, and all $\xi \in \supp(\widehat{\psi})$,
    \begin{align*}
      \norm{\xi-\nu_\psi}_2\leq \alpha \cdot \norm{\xi}_2,
    \end{align*}
    where $\alpha=\alpha(\gamma,\rho):=\sqrt{1-\frac{4\gamma}{(1+\gamma)^2}\cdot (1-\rho)^2}\in (0,1)$.
  \end{lemma}
  
  \begin{proof}
    Since the function $(0,1)\to \R, \ \gamma \mapsto \frac{4\gamma}{(1+\gamma)^2}$ is strictly increasing, we have
    \[0<\frac{4\gamma}{(1+\gamma)^2}\cdot (1-\rho)^2\leq\frac{4\gamma}{(1+\gamma)^2}<\lim_{\gamma \uparrow 1} \frac{4\gamma}{(1+\gamma)^2}= 1.\] 
    Thus, $\alpha \in (0,1)$.
    Now let, for every $\psi \in \Psi_j$, $j \in \N$,
    \[\nu_{\psi}:=\frac{2r_{j}}{1+\gamma} \cdot (1-\rho)\cdot A_{\psi}\nu^* \in \R^d.\]
    By Assumption \ref{ass: high-pass nature and analyticity of filters}, any $\xi \in \supp(\widehat{\psi})$ can be written as $A_{\psi}\omega = \xi$ for some $\omega \in C^\rho \cap \s_{r_{j},r_{j+\kappa}}$. Since $A_{\psi}$ is an orthogonal matrix, it preserves the euclidean distances $\norm{\xi}_2=\norm{\omega}_2$ and
    \[\|\xi-\nu_{\psi}\|_2=\norm{\omega-\frac{2r_{j} \cdot (1-\rho)}{1+\gamma}\cdot \nu^*}_2.\] 
    It therefore suffices to show that the latter term is bounded by $\alpha \cdot \norm{\omega}_2$. By Lemma \ref{lemma: maximum distance to v*},
    \begin{align}
      \frac{\norm{\omega-\frac{2r_{j} \cdot (1-\rho)}{(1+\gamma)}\cdot \nu^*}_2}{\norm{\omega}_2} 
      &\leq \max_{s \in [r_{j}, r_{j+\kappa}]}\norm{\frac{\omega}{\norm{\omega}_2}-\frac{2r_{j} \cdot (1-\rho)}{s \cdot (1+\gamma)}\cdot \nu^*}_2 \nonumber\\
      &\leq  \max_{s \in [r_{j},r_{j+\kappa}]} \max_{x \in C^\rho \cap \mathbb{S}^{d-1}} \Bigg\Vert x-\underbrace{\frac{2r_{j} \cdot (1-\rho)}{s \cdot (1+\gamma)}}_{=:t(s)}\cdot \nu^*\Bigg\Vert_2 \nonumber\\
      &\leq \max_{s \in [r_{j},r_{j+\kappa}]} \sqrt{t(s)^2-2t(s) \cdot (1-\rho)+1} \nonumber\\
      &\leq \max_{s \in [r_{j},\gamma^{-1}r_{j}]} \sqrt{t(s)^2-2t(s) \cdot (1-\rho) +1} \nonumber\\
      &=\max_{t \in \left[\frac{2\gamma \cdot (1-\rho)}{1+\gamma},\frac{2(1-\rho)}{1+\gamma}\right]} \sqrt{t^2-2t \cdot (1-\rho)+1}. \label{maximum of polynomial in t}
    \end{align}
    Since the polynomial \[P:\R \to \R, \quad t \mapsto t^2-2t \cdot (1-\rho)+1\] is convex, the maximum in \eqref{maximum of polynomial in t} is attained at one of the boundary points of the interval. 
    A direct computation yields
    \begin{align*}
      \max_{r \in \left[\frac{2\gamma \cdot (1-\rho)}{1+\gamma},\frac{2(1-\rho)}{1+\gamma}\right]} \sqrt{r^2-2r \cdot (1-\rho)+1} 
      =\sqrt{1-\frac{4\gamma}{(1+\gamma)^2}\cdot (1-\rho)^2}.
    \end{align*}
  \end{proof}
  
  The following lemma is due to Mallat (cf. \cite[Lemma 2.7]{mallat2012group}) and turns out to be essential for proving the main result of this section. For the sake of completeness, we reproduce its (elementary) proof.
  
  \begin{lemma}\label{lem: Extension of Mallat's lemma}
    Let $f\in L^2(\R^d)$, $g \in L^1(\R^d)$, and $\nu \in \R^d$. If $g \geq 0$, then 
    \[\int_{\R^d} \abs{\mathcal{F}(|f|)(\xi)}^2 \cdot \left(1-\abs{\widehat{g}(\xi)}^2\right) \dxi 
    \leq
    \int_{\R^d} \abs{\widehat{f}(\xi)}^2 \cdot \left(1-\abs{\widehat{g}(\xi-\nu)}^2\right) \dxi.\]
  \end{lemma}
  
  \begin{proof}
    Clearly, since $g\geq 0$, we have 
    \begin{align*}
      \norm{|f|*g}_2^2 
      = \norm{|f|*|M_\nu g|}_2^2 
      \geq \norm{f*M_\nu g}_2^2.
    \end{align*}
    By the convolution theorem and Parseval's theorem, we obtain
    \begin{align*}
      \int_{\R^d} \abs{\mathcal{F}(|f|)(\xi)}^2 \cdot \left(1-\abs{\widehat{g}(\xi)}^2\right) \dxi &= \norm{|f|}_2^2-\norm{|f|*g}_2^2 \\
      &\leq \norm{f}_2^2 - \norm{f*M_\nu g}_2^2 =
      \int_{\R^d} \abs{\widehat{f}(\xi)}^2 \cdot \left(1-\abs{\widehat{g}(\xi-\nu)}^2\right) \dxi.
    \end{align*}
  \end{proof}
  
  We are now ready to establish an upper bound on the energy remainder that is of the type \eqref{eq: bound on W_N(f) by integration against K_N}. Previous work in this direction \cite{waldspurger2017exponential,wiatowski2017energy,wiatowski2017topology} showed that, in dimension $d=1$, for certain wavelet-generated filter banks $\Psi_{\text{wav}}$, there exist constants $C=C(\Psi_{\text{wav}})>0$ and $\alpha=\alpha(\Psi_{\text{wav}}) \in (0,1)$ such that, for all (real-valued, cf. \cite{waldspurger2017exponential}) signals $f \in L^2(\R)$, and for all $N \in \N_{\geq 2}$,
  \[W_N[\Psi_{\text{wav}}](f) \leq \int_{\R} |\widehat{f}(\xi)|^2 \cdot \left(1-\left|\widehat{\vartheta}\left(C \cdot  \alpha^{N-1} \cdot \xi\right)\right|^2\right) \dxi.\]
  Here, $\widehat{\vartheta}$ denotes a concretely specified positive definite and even function that ensures asymptotic decay of the right-hand side (a Gaussian in \cite{waldspurger2017exponential}, and a truncated power function in \cite{wiatowski2017energy}). Our result is an extension of this previous work, since it applies to the fairly large class of filter banks satisfying Assumption \ref{ass: high-pass nature and analyticity of filters}, including filter banks that are not necessarily structured as in the case of wavelet-generated filter banks. 
  At the same time, our result holds in arbitrary dimension $d \in \N$. Moreover, for filters that satisfy our Assumption \ref{ass: high-pass nature and analyticity of filters}, which is slightly more restrictive than \cite[Assumption 1]{wiatowski2017energy}, we improve the bound given in \cite[Theorem 1]{wiatowski2017energy} from polynomial decay of order $m_d$ (where $m_d \in (0,1]$ and $m_d \to 0$ as $d\to \infty$) in the argument of $\widehat{\vartheta}$ to exponential decay. 
  
  \begin{theorem}\label{thm: W_N(f) is upper-bounded by integral}
    Let $\vartheta \in L^1(\R^d)$, $\vartheta \geq 0$, and suppose that there exists a nonincreasing function $\eta:[0,\infty)\to [0,1]$ so that $\eta(0)=1$ and $|\widehat{\vartheta}|=\eta(\norm{\cdot}_2)$.
    Then, there exists a universal constant $C_{\chi,\vartheta}>0$ such that, for all $f \in L^2(\R^d)$, and all $N \in \N_{\geq 2}$,
    \begin{align}\label{Upper bound for W_N(f)}
      W_N(f) \leq \int_{\R^d} |\widehat{f}(\xi)|^2 \cdot \left(1-\left|\widehat{\vartheta}\left(C_{\chi,\vartheta} \cdot  \alpha^{N-1} \cdot \xi\right)\right|^2\right) \dxi,
    \end{align}
    where $\alpha=\alpha(\gamma,\rho):=\sqrt{1-\frac{4\gamma}{(1+\gamma)^2}\cdot (1-\rho)^2}$.
    
    Moreover, if there is $\widetilde{C}>0$ so that $|\widehat{\vartheta}(\widetilde{C}\,\cdot)| \leq |\widehat{\chi}|$, then \eqref{Upper bound for W_N(f)} holds for all $N \in \N$ with $C_{\chi,\vartheta}=\widetilde{C}$.
  \end{theorem}
  
  \begin{proof}
    We perform induction on $N$. 
  
    First, let us assume that there is $\widetilde{C}>0$ so that $|\widehat{\vartheta}(\widetilde{C}\;\cdot)| \leq |\widehat{\chi}|$. In this scenario, it is easy to prove the base case, which works analogously to the proof of the base case in \cite[Theorem 1]{wiatowski2017energy}. In fact, by the Littlewood-Paley condition \eqref{ass: Littlewood-Paley condition}, for all $f \in L^2(\R^d)$, 
    \begin{align*}
      W_1(f)=\sum_{\psi \in \Psi} \norm{f*\psi}_2^2 
      &= \int_{\R^d} \sum_{\psi \in \Psi} |\widehat{f}(\xi)|^2 \cdot |\widehat{\psi}(\xi)|^2  \dxi \\
      &= \int_{\R^d} |\widehat{f}(\xi)|^2 \cdot (1-|\widehat{\chi}(\xi)|^2) \dxi
      \leq \int_{\R^d} |\widehat{f}(\xi)|^2 \cdot (1-|\widehat{\vartheta}(\widetilde{C} \cdot \xi)|^2) \dxi.
    \end{align*}
    
    The proof of the base case is harder if there is no $\widetilde{C}>0$ such that $|\widehat{\vartheta}(\widetilde{C}\,\cdot)| \leq |\widehat{\chi}|$ holds true. Here we proceed in three steps, motivated by \cite[Section 6.3]{waldspurger2017exponential}. In the first two steps we construct certain auxiliary functions, whose properties we exploit to eventually prove the base case in the last step.
  
    \textit{Step 1:} Let $r_1>0$ be as in Assumption \ref{ass: high-pass nature and analyticity of filters}, corresponding to the radius of the frequency gap of the high-pass filters $\Psi$. Choose an even function $h \in C^\infty(\R^d)$ that is supported in $\overline{B_{\frac{r_1}{2}}(0)}$ and that satisfies $h(\xi)>0$ for all $\xi \in B_{\frac{r_1}{2}}(0)$. Set $g:=\mathcal{F}(h*h)$. Then, $g \in S(\R^d)$, $g=|\widehat{h}|^2 \geq 0$, $\supp(\widehat{g})\subseteq \overline{B_{r_1}(0)}$, and $\widehat{g}(\xi)>0$ for all $\xi \in \overline{B_{\frac{r_1}{2}}(0)}$. In particular, \[m_g:=\min_{\xi \in \overline{B_{\frac{r_1}{2}}(0)}} |\widehat{g}(\xi)|^2>0.\] 
    Finally, we may assume that $\|\widehat{g}\|_{\infty}\leq 1$ by imposing $\norm{h}_2\leq 1$ if necessary.
  
    \textit{Step 2:} Let $\tilde{c}:[0,\infty) \to (0,\infty)$ be nonincreasing and such that \[\int_{\R^d}\tilde{c}(\norm{\nu}_2)\dnu=1.\] Defining the functions $c:\R^d \to [0,\infty), \ \nu \mapsto \tilde{c}(\norm{\nu}_2)$ and $F:=|\widehat{g}|^2*c$, we find that, for all $\xi \in \R^d$,
    \begin{align*}
      F(\xi)=\int_{\R^d}|\widehat{g}(\nu)|^2 \cdot c(\xi-\nu)\dnu 
      \geq m_g \int_{B_{\frac{r_1}{2}}(0)}c(\xi-\nu)\dnu
      \geq m_g \cdot \vol\left(B_{\frac{r_1}{2}}(0)\right) \cdot \tilde{c}\left(\norm{\xi}_2+\frac{r_1}{2}\right).
    \end{align*}
    Thus, if we choose $c=\tilde{c}(\norm{\cdot}_2)$ so that 
    \[\lim_{s \to \infty}\frac{\eta^2(s)}{\tilde{c}\left(s+\frac{r_1}{2}\right)}=0,\] 
    then there is $M>0$ such that, for all $\xi \in \R^d \setminus B_{M}(0)$, it holds $F(\xi) \geq |\widehat{\vartheta}(\xi)|^2$. Furthermore, $F$ is strictly positive and continuous. Hence, 
    \[m_{F}:=\min_{\xi \in \overline{B_{M}(0)}}F(\xi)>0.\] 
    Since $\lim_{s \to \infty}\eta(s)=0$, there exists $C_{\chi,\vartheta}>\alpha^{-1}$ such that, for all $\xi \in \R^d\setminus B_{r_1}(0)$,
    \[m_{F} \geq |\widehat{\vartheta}(C_{\chi,\vartheta} \cdot \alpha \cdot \xi)|^2.\] 
    Finally, we deduce $\norm{F}_{\infty}\leq 1$ from $\|\widehat{g}\|_{\infty}\leq 1$ and $\norm{c}_1=1$. Altogether, for all $\xi \in \R^d\setminus B_{r_1}(0)$,
    \begin{align}\label{eq: Waldspurger-type bound on 1-F}
      0 \leq 1-F(\xi)\leq 1-|\widehat{\vartheta}(C_{\chi,\vartheta} \cdot \alpha \cdot \xi)|^2.
    \end{align}
    
    \textit{Step 3:}
    Recall from \eqref{frequency gap} that $\sum_{\psi \in \Psi} |\widehat{\psi}|^2$ vanishes on $B_{r_1}(0)$. Since $\widehat{g}$ is supported in $\overline{B_{r_1}(0)}$ and since $\|\widehat{g}\|_{\infty}\leq 1$, we have 
    \[\sum_{\psi \in \Psi} |\widehat{\psi}(\xi)|^2\leq 1-|\widehat{g}(\xi)|^2 \quad \text{ a.e. } \xi \in \R^d.\]
    Hence, for all $f \in L^2(\R^d)$,
    \begin{align*}
      W_2(f)=\sum_{\psi \in \Psi}\sum_{\psi^\prime \in \Psi} \norm{||f*\psi|*\psi^\prime|}_2^2 &= \sum_{\psi \in \Psi}\sum_{\psi^\prime \in \Psi} \norm{|f*\psi|*\psi^\prime}_2^2 \\
      &\leq \sum_{\psi \in \Psi} \int_{\R^d} \abs{\mathcal{F}(|f*\psi|)(\xi)}^2 \cdot \left(1-|\widehat{g}(\xi)|^2\right) \dxi. 
    \end{align*}
    Using \eqref{frequency gap} and Lemma \ref{lem: Extension of Mallat's lemma}, we conclude that, for all $\nu \in \R^d$,
    \begin{align*}
      W_2(f)&\leq \sum_{\psi \in \Psi} \int_{\R^d} |\widehat{f}(\xi)|^2 \cdot |\widehat{\psi}(\xi)|^2 \cdot \left(1-|\widehat{g}(\xi-\nu)|^2\right) \dxi \\
      &\leq \int_{\R^d\setminus B_{r_1}(0)} |\widehat{f}(\xi)|^2 \cdot \left(1-|\widehat{g}(\xi-\nu)|^2\right) \dxi.
    \end{align*}
    Finally, we average this inequality with the function $c$ and insert the upper bound from \eqref{eq: Waldspurger-type bound on 1-F} to obtain
    \begin{align*}
      W_2(f)&\leq \int_{\R^d} c(\nu) \cdot \int_{\R^d\setminus B_{r_1}(0)} |\widehat{f}(\xi)|^2 \cdot \left(1-|\widehat{g}(\xi-\nu)|^2\right) \dxi \dnu \\
      &= \int_{\R^d\setminus B_{r_1}(0)} |\widehat{f}(\xi)|^2 \cdot \int_{\R^d} c(\nu)-c(\nu) \cdot |\widehat{g}(\xi-\nu)|^2  \dnu \dxi \\
      &= \int_{\R^d\setminus B_{r_1}(0)} |\widehat{f}(\xi)|^2 \cdot \left(1-F(\xi)\right) \dxi\\
      &\leq \int_{\R^d} |\widehat{f}(\xi)|^2 \cdot \left(1-|\vartheta(C_{\chi,\vartheta} \cdot \alpha \cdot \xi)|^2\right) \dxi.
    \end{align*}
    This concludes the proof of the base case.
  
    The induction step goes along the lines of \cite[Theorem 1]{wiatowski2017energy}: Suppose that \eqref{Upper bound for W_N(f)} holds for some $N \in \N$ for all $f\in L^2(\R^d)$. Our goal is then to establish the analogous upper bound for $W_{N+1}(f)$. 
    For $\psi \in \Psi$, let $\nu_\psi \in \R^d$ be as in Lemma \ref{lem: how alpha comes into play}. Note that the inverse Fourier transform of $\widehat{\vartheta}\left(C_{\chi,\vartheta} \cdot \alpha^{N-1}\,\cdot\right)$ is nonnegative by the non-negativity of $\vartheta$. Thus, applying the induction hypothesis to $|f*\psi|$, as well as employing Lemma \ref{lem: Extension of Mallat's lemma}, we find that
    \begin{align*}
      W_{N}(|f*\psi|)
      &\leq \int_{\R^d} |\mathcal{F}(|f*\psi|)(\xi)|^2 \cdot \left(1-\left|\widehat{\vartheta}\left(C_{\chi,\vartheta} \cdot \alpha^{N-1} \cdot \xi\right)\right|^2\right) \dxi \nonumber\\
      &\leq \int_{\R^d} |\widehat{f}(\xi)|^2 \cdot |\widehat{\psi}(\xi)|^2 \cdot \left(1-\left|\widehat{\vartheta}\left(C_{\chi,\vartheta} \cdot \alpha^{N-1} \cdot (\xi-\nu_\psi)\right)\right|^2\right)\dxi.
    \end{align*}
    Summing over $\psi \in \Psi$ on both sides of the above inequality yields
    \begin{align}\label{eq: W_{N+1}(f) in terms of W_{N}(f)}
      W_{N+1}(f)=\sum_{\psi \in \Psi} W_{N}(|f*\psi|)\leq \int_{\R^d} |\widehat{f}(\xi)|^2 \cdot h_{N-1}(\xi) \dxi,
    \end{align}
    where
    \begin{align*}
      h_{N-1}(\xi):=\sum_{\psi \in \Psi} |\widehat{\psi}(\xi)|^2 \cdot \left(1-\left|\widehat{\vartheta}\left(C_{\chi,\vartheta} \cdot \alpha^{N-1} \cdot (\xi-\nu_\psi)\right)\right|^2\right).
    \end{align*}
  
    Now, Lemma \ref{lem: how alpha comes into play} comes into play to establish a pointwise upper bound for $h_{N-1}$ that suffices to conclude the proof. Since $\sum_{\psi \in \Psi} |\widehat{\psi}(\xi)|^2 \leq 1$ holds for all $\xi \in \R^d$, it is enough to show that we have, for all $\psi \in \Psi$, and all $\xi \in \supp(\widehat{\psi})$,
    \begin{align*}
      1-\left|\widehat{\vartheta}\left(C_{\chi,\vartheta} \cdot \alpha^{N-1} \cdot (\xi-\nu_\psi)\right)\right|^2 \leq 1-\left|\widehat{\vartheta}\left(C_{\chi,\vartheta} \cdot \alpha^{N} \cdot \xi\right)\right|^2.
    \end{align*}
    In fact, $1-|\widehat{\vartheta}|^2=1-\eta^2(\norm{\cdot}_2)$, where $1-\eta^2$ is nondecreasing, and Lemma \ref{lem: how alpha comes into play} implies that
    \[\norm{C_{\chi,\vartheta} \cdot \alpha^{N-1} \cdot (\xi-\nu_\psi)}_2 \leq \norm{C_{\chi,\vartheta} \cdot \alpha^{N} \cdot \xi}_2.\]
    Consequently, we have, for all $\xi \in \R^d$,
    \[h_{N-1}(\xi)\leq 1-\left|\widehat{\vartheta}\left(C_{\chi,\vartheta} \cdot \alpha^{N} \cdot \xi\right)\right|^2.\]
    Inserting this into \eqref{eq: W_{N+1}(f) in terms of W_{N}(f)} concludes the induction step and thus the proof of the theorem.
  \end{proof}

  \begin{remark}
    Smaller values of $\alpha$ lead to faster convergence to $0$ of the integral on the right side of \eqref{Upper bound for W_N(f)} as $N \to \infty$. Hence, let us briefly comment on the impact of the parameters on the value of $\alpha$. To this end, we note that the function 
    \[\alpha:(0,1)\times [0,1) \to (0,1), \quad (\gamma,\rho)\mapsto \sqrt{1-\frac{4\gamma}{(1+\gamma)^2}\cdot (1-\rho)^2}\]
    is strictly decreasing in $\gamma$, while it is strictly increasing in $\rho$. Moreover, we have the asymptotic behavior
    \[\lim_{(\gamma,\rho)\to (1,0)} \alpha(\gamma,\rho)=0.\]
    We interpret this as follows: $\alpha$ is small if the size of the Fourier supports of the filters is approximately constant across all scales and if the filters are well localized in terms of their maximum angular frequency spread. This result appears to be consistent with the earlier findings stated in \cite[Theorem 1]{wiatowski2017energy}, which indicate a faster decay in lower dimensions: A reduction of the underlying dimension implicitly forces a higher concentration of angular frequency, with a maximum concentration achieved in dimension $d=1$ (Fig.~\ref{fig: plot of alpha}).
  \end{remark}
  
  \begin{figure}[h]
    \centering
    \includegraphics{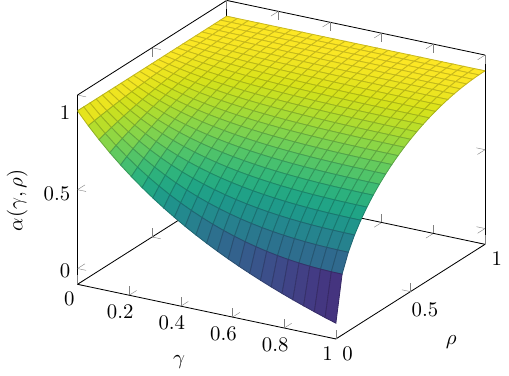}
    \caption{Plot of $\alpha$ as a function of the parameters $\gamma$ and $\rho$.
    Smaller values of $\alpha$ guarantee faster decay.} 
    \label{fig: plot of alpha}
  \end{figure}
  
  Next, we show how Theorem \ref{thm: W_N(f) is upper-bounded by integral} entails fast (up to exponential) energy decay for signals belonging to generalized Sobolev spaces that are tailored to the frequency localization of the high-pass filters $\Psi$. We introduce certain weights to define these spaces.
  
  \begin{definition}
    Let $k>0$. We say that a function $\omega:(0,\infty)\to (0,\infty)$ is a \textbf{weakly $t^k$-dominated weight} if $\omega$ is nondecreasing and if there exists $T>0$ so that the auxiliary function
    \[h_{k,\omega}:(0,\infty) \to (0,\infty), \ t \mapsto t^k \cdot \omega^{-2}(t)\] 
    is bounded on $(0,T)$, nondecreasing on $[T,\infty)$, and satisfies $\lim_{t \to \infty}h_{k,\omega}(t)=\infty$. 
  
    If, in addition to the above, $h_{k,\omega}$ is nondecreasing on the entire interval $(0,\infty)$, then we say that $\omega$ is a \textbf{strongly $t^k$-dominated weight}.
  \end{definition}
  
  \begin{remark}\label{rm: Sufficient condition if the weight is differentiable}
    If $\omega$ is differentiable, we can easily establish a sufficient criterion for the (simultaneous) monotonicity of $\omega$ and $h_{k,\omega}$, relying on the non-negativity of their derivatives. In fact, for all $t>0$, 
    \begin{align*}
      h_{k,\omega}^\prime(t)=k\cdot t^{k-1} \cdot \omega^{-2}(t) - 2 t^k \cdot \omega^{-3}(t) \cdot \omega^\prime(t)
      = \omega^{-2}(t) \cdot t^{k-1} \cdot \left(k-2t\cdot \omega^{-1}(t) \cdot \omega^\prime(t)\right).
    \end{align*}
    Thus, $h_{k,\omega}^\prime(t)\geq 0$ is equivalent to $k\cdot \omega(t)\geq 2 t \cdot \omega^\prime(t)$. Altogether, if there is $T\geq 0$ so that for all $t>T$, 
    \[k\cdot \omega(t)\geq 2 t \cdot \omega^\prime(t)\geq 0,\]
    then both $\omega$ and $h_{k,\omega}$ are nondecreasing on $(T,\infty)$.
  \end{remark}
  
  We define, for all $\psi \in \Psi$,
  \begin{align}\label{def: diameter of spectral support of a filter}
    d_\psi:=\inf_{\xi^\prime \in \R^d}\sup_{\xi \in \supp(\widehat{\psi})} \norm{\xi-\xi^\prime}_2=\min_{\xi^\prime \in \R^d}\max_{\xi \in \supp(\widehat{\psi})} \norm{\xi-\xi^\prime}_2.
  \end{align}
  For a weakly $t^k$-dominated weight, we consider the generalized Sobolev space 
  \begin{align}\label{def: generalized Sobolev space}
    \mathcal{D}_\omega(\Psi;L^2(\R^d)):=\set{f \in L^2(\R^d)~\middle|~ \sum_{\psi \in \Psi} \omega^2(d_\psi) \cdot \norm{f*\psi}_2^2< \infty}.
  \end{align}

  \begin{theorem}\label{thm: Convergence rates for W_N(f)}
    Let $\vartheta \in L^1(\R^d)$ be as in Theorem \ref{thm: W_N(f) is upper-bounded by integral}, and assume that there are $k=k_{\vartheta}>0$, $C=C_{\vartheta}>0$ such that, for all $\xi \in \R^d$, 
    \begin{align}\label{ass: order of convergence for 1-|vartheta|^2}
      1-|\widehat{\vartheta}(\xi)|^2\leq C \cdot \norm{\xi}_2^{k}.
    \end{align}
    If $\omega:(0,\infty)\to (0,\infty)$ is a weakly $t^k$-dominated weight, then
    we have, for all $f \in \mathcal{D}_\omega(\Psi;L^2(\R^d))$,  
    \[W_N(f) \in \mathcal{O}\left(\omega^{-2}\left(\alpha^{-N}\right)\right),\]
    where $\alpha=\alpha(\gamma,\rho):=\sqrt{1-\frac{4\gamma}{(1+\gamma)^2}\cdot (1-\rho)^2} \in (0,1)$.
  
    Moreover, if there is $\widetilde{C}>0$ so that $|\widehat{\vartheta}(\widetilde{C} \, \cdot)| \leq |\widehat{\chi}|$, and if $\omega$ is a strongly $t^k$-dominated weight, then the following explicit upper bound holds: For all $N \in \N_{\geq 2}$,
    \begin{align*}
      W_N(f)\leq \max\left\{1,C \cdot \widetilde{C}^k \cdot \alpha^{-k}\right\} \cdot \left(\sum_{\psi \in \Psi} \omega^2(d_\psi) \cdot \norm{f*\psi}_2^2\right) \cdot  \omega^{-2}(\alpha^{-N}).
    \end{align*}
  \end{theorem}
  
  \begin{remark}
    By the non-negativity of $\vartheta$, the condition \eqref{ass: order of convergence for 1-|vartheta|^2} can only be satisfied if $k=k_\vartheta\leq 2$. Let us briefly sketch the reason for this. By definition of the $L^1$-Fourier transform, we have, for all $\xi \in \R^d$,
    \begin{align*}
      1-|\widehat{\vartheta}(\xi)|^2&= \re\left(|\widehat{\vartheta}(0)|^2-|\widehat{\vartheta}(\xi)|^2\right)\\
      &= \re\left(\int_{\R^{2d}} \vartheta(x)\cdot \vartheta(y) \cdot \left(1-e^{-2\pi \cdot i \cdot \ip{\xi,x-y}}\right)\dxy\right) \\
      &= \int_{\R^{2d}} \vartheta(x)\cdot \vartheta(y) \cdot \left(1-\cos(2\pi \cdot \ip{\xi,x-y})\right)\dxy.
    \end{align*} 
    For all $z \in [-1,1]$, 
    \[1-\cos(z)\geq \frac{z^2}{3}.\] 
    Hence, if $\xi=s\cdot e_1$, where $s>0$, and $e_1$ denotes the first standard unit vector in $\R^d$, we obtain
    \begin{align*}
      \frac{1-|\widehat{\vartheta}(\xi)|^2}{\norm{\xi}_2^2} &\geq \int_{\set{(x,y) \in \R^{2d}~:~|2\pi\cdot s\cdot (x_1-y_1)|\leq 1}} \vartheta(x)\cdot \vartheta(y) \cdot (x_1-y_1)^2 \dxy\\
      &\xrightarrow{s \to 0} \int_{\R^{2d}} \vartheta(x)\cdot \vartheta(y) \cdot \left(x_1-y_1\right)^2\dxy>0.
    \end{align*}
    We conclude that $k_\vartheta \leq 2$.

    Now, observe that larger values for $k$ in \eqref{ass: order of convergence for 1-|vartheta|^2} allow more flexibility in the choice of $\omega$, as the assumptions regarding $h_{k,\omega}$ are easier to fulfill. However, as justified above, the largest possible value for any $\vartheta \in L^1(\R^d)$ is $k_\vartheta=2$. Finally, note that, while we can show an analogous result to Theorem \ref{thm: W_N(f) is upper-bounded by integral} for $\vartheta \in L^2(\R^d)$, we would not gain a qualitative improvement to Theorem \ref{thm: Convergence rates for W_N(f)}, since the same threshold order of convergence $k_\vartheta=2$ in \eqref{ass: order of convergence for 1-|vartheta|^2} would still hold.
  \end{remark}
  
  \begin{proof}[Proof of Theorem \ref{thm: Convergence rates for W_N(f)}]
    Fix $f \in \mathcal{D}_\omega(\Psi;L^2(\R^d))$. Let, for every $\psi \in \Psi$, 
    \[\xi_\psi \in \argmin_{\xi^\prime \in \R^d}\max_{\xi \in \supp(\widehat{\psi})} \norm{\xi-\xi^\prime}_2.\] 
    We start in a similar fashion to the induction step of the proof of Theorem \ref{thm: W_N(f) is upper-bounded by integral}. By the same theorem, by the path structure of the scattering network, and by Lemma \ref{lem: Extension of Mallat's lemma}, we find that, for all sufficiently large $N \in \N$,
    \begin{align}\label{eq: (9)}
      W_{N}(f)&=\sum_{\psi \in \Psi} W_{N-1}(|f*\psi|) \nonumber\\
      &\leq \sum_{\psi \in \Psi} \int_{\R^d} \abs{\mathcal{F}(|f*\psi|)(\xi)}^2 \cdot \left(1-\abs{\widehat{\vartheta}\left(C_{\chi,\vartheta}\cdot \alpha^{N-1}\cdot \xi\right)}^2\right) \dxi \nonumber\\
      &\leq \sum_{\psi \in \Psi} \int_{\R^d} |\widehat{f}(\xi)|^2 \cdot |\widehat{\psi}(\xi)|^2 \cdot \left(1-\abs{\widehat{\vartheta}\left(C_{\chi,\vartheta}\cdot \alpha^{N-1}\cdot (\xi-\xi_\psi)\right)}^2\right) \dxi.
    \end{align}
    
    Our strategy is to split the integral into small scales and large scales compared with $\alpha^{-N}$, and to establish upper bounds on those terms separately. Note that our conditions on $h_{k,\omega}$ guarantee that, for all sufficiently large $N \in \N$, 
    \begin{align}\label{eq: (8)}
      \sup_{t \in (0,\alpha^{-N}]} h_{k,\omega}(t)= h_{k,\omega}(\alpha^{-N})=\alpha^{-kN}\cdot \omega^{-2}(\alpha^{-N}).
    \end{align}
    For the small scales, we use \eqref{ass: order of convergence for 1-|vartheta|^2} and \eqref{eq: (8)} to derive
    \begin{align*}
      \int_{B_{\alpha^{-N}}(\xi_\psi)} &|\widehat{f}(\xi)|^2 \cdot |\widehat{\psi}(\xi)|^2 \cdot \left(1-\abs{\widehat{\vartheta}\left(C_{\chi,\vartheta}\cdot \alpha^{N-1}\cdot (\xi-\xi_\psi)\right)}^2\right) \dxi\\
      &\leq \int_{B_{\alpha^{-N}}(\xi_\psi)} |\widehat{f}(\xi)|^2 \cdot |\widehat{\psi}(\xi)|^2 \cdot C \cdot \norm{C_{\chi,\vartheta} \cdot \alpha^{N-1} \cdot (\xi-\xi_\psi)}_2^k \dxi \\
      &= C \cdot C_{\chi,\vartheta}^k \cdot  \alpha^{k(N-1)} \cdot \int_{B_{\alpha^{-N}}(\xi_\psi)} |\widehat{f}(\xi)|^2 \cdot |\widehat{\psi}(\xi)|^2 \cdot \omega^{2}(\norm{\xi-\xi_\psi}_2) \cdot h_{k,\omega}(\norm{\xi-\xi_\psi}_2) \dxi \\
      &\leq C \cdot C_{\chi,\vartheta}^k \cdot \alpha^{-k} \cdot \omega^{-2}(\alpha^{-N}) \cdot \int_{B_{\alpha^{-N}}(\xi_\psi)} |\widehat{f}(\xi)|^2 \cdot |\widehat{\psi}(\xi)|^2 \cdot \omega^{2}(d_\psi) \dxi.
    \end{align*}
    The monotonicity of $\omega$ suffices to control the large scales,
    \begin{align*}
      \int_{\R^d\setminus B_{\alpha^{-N}}(\xi_\psi)} &|\widehat{f}(\xi)|^2 \cdot |\widehat{\psi}(\xi)|^2 \cdot \left(1-\abs{\widehat{\vartheta}\left(C_{\chi,\vartheta}\cdot \alpha^{N-1}\cdot (\xi-\xi_\psi)\right)}^2\right) \dxi \\
      &\leq \int_{\R^d\setminus B_{\alpha^{-N}}(\xi_\psi)} |\widehat{f}(\xi)|^2 \cdot |\widehat{\psi}(\xi)|^2 \cdot \omega^{2}(\norm{\xi-\xi_\psi}_2) \cdot \omega^{-2}(\norm{\xi-\xi_\psi}_2) \dxi \\
      &\leq \omega^{-2}(\alpha^{-N}) \cdot \int_{\R^d\setminus B_{\alpha^{-N}}(\xi_\psi)} |\widehat{f}(\xi)|^2 \cdot |\widehat{\psi}(\xi)|^2 \cdot \omega^{2}(d_\psi) \dxi.
    \end{align*}
    Inserting those estimates into \eqref{eq: (9)} yields, for all sufficiently large $N \in \N$, 
    \begin{align*}
      W_{N}(f) &\leq \max\left\{1,C \cdot C_{\chi,\vartheta}^k \cdot \alpha^{-k}\right\} \cdot \omega^{-2}(\alpha^{-N}) \cdot \sum_{\psi \in \Psi} \int_{\R^d} |\widehat{f}(\xi)|^2 \cdot |\widehat{\psi}(\xi)|^2 \cdot \omega^{2}(d_\psi) \dxi \\
      &= \max\left\{1,C \cdot C_{\chi,\vartheta}^k \cdot \alpha^{-k}\right\} \cdot \left(\sum_{\psi \in \Psi} \omega^{2}(d_\psi) \cdot \norm{f*\psi}_2^2\right) \cdot \omega^{-2}(\alpha^{-N}).
    \end{align*}
    Now, if there is $\widetilde{C}>0$ so that $|\widehat{\vartheta}(C_{\chi,\vartheta} \, \cdot)| \leq |\widehat{\chi}|$ and if $h_{k,\omega}$ is nondecreasing on the entire interval $(0,\infty)$, then \eqref{eq: (9)} and \eqref{eq: (8)}, and thus also the latter bound for $W_N(f)$, are in fact valid for all $N \in \N_{\geq 2}$.
    This concludes the proof of the theorem.
  \end{proof}
  
  The following lemma shows that fast decay (with respect to the weight $\omega$) of the Fourier transform suffices, independent of $\Psi$, for a signal to belong to the generalized Sobolev space $\mathcal{D}_\omega(\Psi;L^2(\R^d))$.
  
  \begin{lemma}\label{lem: Inclusion relation between Fourier-weighted L2 space and decomp space}
    Let $k>0$. If $\omega:(0,\infty)\to (0,\infty)$ is a weakly $t^k$-dominated weight, then 
    \[\mathcal{F}L_\omega^2(\R^d) \subseteq \mathcal{D}_\omega(\Psi;L^2(\R^d)).\] 
  \end{lemma}
  
  \begin{proof}
    Recall from Assumption \ref{ass: high-pass nature and analyticity of filters} that $\Psi$ can be decomposed into $\Psi=\bigcup_{j\in \N} \Psi_j$ such that $\supp(\widehat{\psi})\subseteq \mathcal{S}_{r_j,r_{j+\kappa}}$ whenever $\psi \in \Psi_j$. In particular, $d_\psi \leq r_{j+\kappa}\leq \gamma^{-1}\cdot r_j$. 
    By assumption on $\omega$, there exists $J \in \N$ such that $h_{k,\omega}$ 
    is nondecreasing on $(r_J,\infty)$. Hence, for all $j \in \N_{>J}$, 
    \begin{align}\label{eq: (10)}
      \omega^2(d_\psi) \leq \omega^2(\gamma^{-1}\cdot r_j) \leq \gamma^{-k} \cdot \omega^2(r_j).
    \end{align}
    Let $f \in \mathcal{F}L_\omega^2(\R^d)$. We aim to show that the sum $\sum_{\psi \in \Psi} \omega^{2}(d_\psi) \cdot \norm{f*\psi}_2^2$ converges. To do so, we split the sum into small scales and large scales compared with $J$,
    \begin{align*}
      \sum_{j=1}^J \sum_{\psi \in \Psi_j} \int_{\mathcal{S}_{r_j,r_{j+\kappa}}} |\widehat{f}(\xi)|^2 \cdot |\widehat{\psi}(\xi)|^2 \cdot \omega^{2}(d_\psi) \dxi + \sum_{j=J+1}^\infty \sum_{\psi \in \Psi_j} \int_{\mathcal{S}_{r_j,r_{j+\kappa}}} |\widehat{f}(\xi)|^2 \cdot |\widehat{\psi}(\xi)|^2 \cdot \omega^{2}(d_\psi) \dxi,
    \end{align*}
    so that it suffices to bound these expressions individually. We can easily control the small scales by the monotonicity of $\omega$, 
    \begin{align*}
      \sum_{j=1}^J \sum_{\psi \in \Psi_j} \int_{\mathcal{S}_{r_j,r_{j+\kappa}}} &|\widehat{f}(\xi)|^2 \cdot |\widehat{\psi}(\xi)|^2 \cdot \omega^{2}(d_\psi) \dxi \\
      &\leq \omega^{2}(r_{J+\kappa}) \cdot \sum_{j=1}^J \sum_{\psi \in \Psi_j} \int_{\mathcal{S}_{r_j,r_{j+\kappa}}} |\widehat{f}(\xi)|^2 \cdot |\widehat{\psi}(\xi)|^2 \dxi \\
      &\leq \frac{\omega^{2}(r_{J+\kappa})}{\omega^2(r_1)} \cdot \sum_{j=1}^J \sum_{\psi \in \Psi_j} \int_{\mathcal{S}_{r_j,r_{j+\kappa}}} |\widehat{f}(\xi)|^2 \cdot |\widehat{\psi}(\xi)|^2 \cdot \omega^2(\norm{\xi}_2) \dxi.
    \end{align*} 
    For the large scales, in addition to the monotonicity of $\omega$, we also use \eqref{eq: (10)} to find that
    \begin{align*}
      \sum_{j=J+1}^\infty \sum_{\psi \in \Psi_j} \int_{\mathcal{S}_{r_j,r_{j+\kappa}}} &|\widehat{f}(\xi)|^2 \cdot |\widehat{\psi}(\xi)|^2 \cdot \omega^{2}(d_\psi) \dxi \\
      &\leq \sum_{j=J+1}^\infty \sum_{\psi \in \Psi_j} \omega^{2}(\gamma^{-1}\cdot r_{j}) \cdot \int_{\mathcal{S}_{r_j,r_{j+\kappa}}} |\widehat{f}(\xi)|^2 \cdot |\widehat{\psi}(\xi)|^2 \dxi \\
      &\leq \sum_{j=J+1}^\infty \sum_{\psi \in \Psi_j} \gamma^{-k} \cdot \omega^{2}(r_{j}) \cdot \int_{\mathcal{S}_{r_j,r_{j+\kappa}}} |\widehat{f}(\xi)|^2 \cdot |\widehat{\psi}(\xi)|^2 \dxi \\
      &\leq \sum_{j=J+1}^\infty \sum_{\psi \in \Psi_j} \gamma^{-k} \cdot \int_{\mathcal{S}_{r_j,r_{j+\kappa}}} |\widehat{f}(\xi)|^2 \cdot |\widehat{\psi}(\xi)|^2 \cdot \omega^2(\norm{\xi}_2) \dxi.
    \end{align*}
    Altogether, letting $C=\max\set{\frac{\omega^{2}(r_{J+\kappa})}{\omega^2(r_1)}, \gamma^{-k}}$, we obtain
    \begin{align*}
      \sum_{\psi \in \Psi} \omega^{2}(d_\psi) \cdot \norm{f*\psi}_2^2 &\leq C \cdot \sum_{j=1}^\infty \sum_{\psi \in \Psi_j} \int_{\mathcal{S}_{r_j,r_{j+\kappa}}} |\widehat{f}(\xi)|^2 \cdot |\widehat{\psi}(\xi)|^2 \cdot \omega^2(\norm{\xi}_2) \dxi \\
      &\leq C \cdot \kappa \cdot \sum_{\psi \in \Psi} \int_{\R^d} |\widehat{f}(\xi)|^2 \cdot |\widehat{\psi}(\xi)|^2 \cdot \omega^2(\norm{\xi}_2) \dxi \\
      &\leq  C \cdot \kappa \cdot \int_{\R^d} |\widehat{f}(\xi)|^2 \cdot \omega^2(\norm{\xi}_2) \dxi < \infty.
    \end{align*}
  \end{proof}
  
  There are choices of the weight $\omega$ and the filter bank $\Psi$, for which the inclusion $\mathcal{F}L_\omega^2(\R^d)\subseteq \mathcal{D}_\omega(\Psi;L^2(\R^d))$ is strict, i.e., the spaces $\mathcal{F}L_\omega^2(\R^d)$ and $\mathcal{D}_\omega(\Psi;L^2(\R^d))$ do not coincide, see Remark \ref{rm: strict inclusion between Fourier weighted L2 and decomp space}. 
  
  As a consequence of Theorem \ref{thm: Convergence rates for W_N(f)} and Lemma \ref{lem: Inclusion relation between Fourier-weighted L2 space and decomp space}, we obtain asymptotic convergence rates for signals of (logarithmic) Sobolev regularity.
  
  \begin{corollary}\label{cor: Decay rates for (Log)Sobolev functions}
    Suppose that Assumption \ref{ass: high-pass nature and analyticity of filters} holds for $\Psi$ with parameters $\gamma$ and $\rho$. 
    As before, let 
    \[\alpha:=\sqrt{1-\frac{4\gamma}{(1+\gamma)^2}\cdot (1-\rho)^2} \: \in (0,1).\] 
    Then, the following hold:
    \begin{enumerate}[a)]
      \item If $f \in H^s(\R^d)$ for some $s>0$, then 
      \[W_N(f) \in \mathcal{O}(\alpha^{2\min\set{s,1}\cdot N}).\] 
      \item If $f \in H^s_{\log}(\R^d)$ for some $s>0$, then 
      \[W_N(f)\in \mathcal{O}(N^{-2s}).\] 
    \end{enumerate}
  \end{corollary}
  
  \begin{proof}
    Let \[\vartheta: \R^d \to \R, \quad x \mapsto e^{-\pi \cdot \norm{x}_2^2}.\] 
    Then, $\vartheta$ satisfies the prerequisites of Theorem \ref{thm: W_N(f) is upper-bounded by integral}. From the series expansion of the one-dimensional Gaussian at $0$ we can see that, for all $\xi \in \R^d$, 
    \[1-|\widehat{\vartheta}(\xi)|^2\leq 2 \pi \cdot \norm{\xi}_2^2.\]
    Now, part a) follows immediately from Theorem \ref{thm: Convergence rates for W_N(f)} and Lemma \ref{lem: Inclusion relation between Fourier-weighted L2 space and decomp space} if we take 
    \[\omega: (0,\infty)\to (0,\infty), \quad t \mapsto (1+t^2)^{\frac{s}{2}}.\] 
    In fact, using the criterion from Remark \ref{rm: Sufficient condition if the weight is differentiable}, it is straightforward to see that $\omega$ is a strongly $t^2$-dominated weight if $s \in (0,1]$, since for all $t>0$,
    \[2\omega(t)\geq 2s \cdot \frac{t^2}{1+t^2} \cdot (1+t^2)^{\frac{s}{2}} = 2t \cdot \omega^\prime(t)\geq 0.\]  
    We draw the conclusion by noting that $\mathcal{F}L_\omega^2(\R^d)=H^s(\R^d)$ and that, for all $N \in \N$, 
    \[\omega^{-2}(\alpha^{-N})\leq \alpha^{2s\cdot N}.\] 
    For $s>1$, the statement simply follows from the inclusion $H^s(\R^d)\subseteq H^{1}(\R^d)$.
  
    In order to establish part b), 
    let us consider
    \[\omega: (0,\infty)\to (0,\infty), \quad t \mapsto \ln^{s}(e+t).\]
    We verify the condition from Remark \ref{rm: Sufficient condition if the weight is differentiable} again. In general, $\omega$ is not a strongly $t^2$-dominated weight (if $s$ is too large). However, 
    $\omega$ is a weakly $t^2$-dominated weight. Indeed, for all $t\geq e^{s}-e$, we have $\ln(e+t)\geq s$, which entails
    \[2\omega(t)= 2\ln(e+t) \cdot \ln^{s-1}(e+t) \geq  2s \cdot \frac{t}{e+t} \cdot \ln^{s-1}(e+t) = 2t \cdot \omega^\prime(t)\geq 0.\] 
    Finally, by definition, 
    $\mathcal{F}L_\omega^2(\R^d)=H^s_{\log}(\R^d)$, and we have, for all $N \in \N$,
    \[\omega^{-2}(\alpha^{-N})\leq \ln^{-2s}(\alpha^{-N})=\ln^{-2s}(\alpha^{-1}) \cdot N^{-2s}.\]
    This concludes the proof of part b) by Theorem \ref{thm: Convergence rates for W_N(f)}.
  \end{proof}
  
  \begin{remark}
  	The Sobolev spaces contain many practically relevant signal classes, such
  	as the class of bandlimited signals and the class of cartoon functions.
  	In the mathematical signal processing literature (e.g.,
  	\cite{wiatowski2017energy,wiatowski2016discrete,grohs2016deep} in the context
  	of scattering transforms), cartoon functions often serve as a reference
  	model for natural images (e.g., images of handwritten digits). If defined
  	as in \cite[Definition 1]{wiatowski2017energy}, the cartoon functions indeed
  	belong to $H^{s}({\mathbb{R}}^{d})$ for all $s \in (0,\frac{1}{2})$; see
  	\cite[Lemma 1]{wiatowski2017energy}. Thus, by Corollary~\ref{cor: Decay rates for (Log)Sobolev functions}, their scattering coefficients
  	admit exponential energy decay if the scattering filters satisfy Assumption~\ref{ass: high-pass nature and analyticity of filters}.
  	
    Using arguments very similar to those in the proof of Theorem \ref{thm: Convergence rates for W_N(f)} together with the \textbf{compactly} supported positive definite radial basis function 
    \begin{align}\label{def: Euclids hat}
      \widehat{\vartheta}: \R^d \to \R, \quad \xi \mapsto \left(1-\norm{\xi}_2\right)_+^{\left\lfloor\frac{d}{2}\right\rfloor+1}
    \end{align}
    in the above proof, one can derive the following \textbf{explicit} upper bounds (as opposed to the asymptotic bounds in the above corollary):
    \begin{enumerate}[a)]
      \item If $f \in H^s(\R^d)$ for some $s>0$, then we have, for all $N \in \N$,
      \[W_N(f) \leq \max\set{1,\frac{2\left(\left\lfloor\frac{d}{2}\right\rfloor+1\right)}{\alpha \cdot r_1}} \cdot \norm{f}_{H^s(\R^d)}^2 \cdot \alpha^{\min\{2s,1\} \cdot N}.\] 
      \item If $f \in H^s_{\log}(\R^d)$ for some $s>0$, then we have, for all $N \in \N$, 
      \[W_N(f) \leq \max\set{1,\frac{2\left(\left\lfloor\frac{d}{2}\right\rfloor+1\right)}{\alpha \cdot r_1}} \cdot \ln^{-\min\{2s,1\}}(\alpha^{-1}) \cdot \norm{f}_{H^s_{\log}(\R^d)}^2 \cdot N^{-\min\{2s,1\}}.\] 
    \end{enumerate}
    The differences in the asymptotic rates of decay (compared with the above corollary) are due to the fact that
    \begin{itemize}
      \item the largest possible values $k_\vartheta$ in \eqref{ass: order of convergence for 1-|vartheta|^2} differ if $\vartheta$ is a Gaussian ($k_\vartheta=2$) or if $\vartheta$ is defined as in \eqref{def: Euclids hat} ($k_\vartheta=1$).
      \item the weight $\omega: (0,\infty)\to (0,\infty), t \mapsto \ln^{s}(e+t)$ is not strongly $t$-dominated if $s$ is too large. However, it is easy to see that $\omega$ is strongly $t$-dominated if $s \in (0,\frac{1}{2}]$.
    \end{itemize}
  \end{remark}

  \section{Applications}\label{sec: Applications}
  
  \subsection{Scattering with UFC filters}\label{sec: Scattering with UFC filters}
  
  In this section, we apply our results from Section \ref{sec: Convergence rates for scattering propagation} to a large class of filter banks that is closely related to the class of uniform covering frames introduced in \cite{czaja2019analysis}. 
  
  \begin{definition}\label{def: uniform frequency concentration}
    We say that the filters $\Psi$ have \textbf{uniform frequency concentration} (UFC) if they fulfill Assumption \ref{ass: high-pass nature and analyticity of filters}, and if 
    \[D_\Psi:=\sup_{\psi \in \Psi} d_\psi=\sup_{\psi \in \Psi} \inf_{\xi_\psi \in \R^d}\sup_{\xi \in \supp(\widehat{\psi})} \norm{\xi-\xi_\psi}_2<\infty.\]
  \end{definition}
  
  Starting from \cite[Proposition 2.3]{czaja2019analysis}, one can see that
  the class of uniform frequency concentration filters includes certain Weyl-Heisenberg
  (Gabor) frames. More generally, UFC filters are easy to construct as follows:
  Departing from a covering of the frequency space ${\mathbb{R}}^{d}$ that
  is compatible with Assumption~\ref{ass: high-pass nature and analyticity of filters}, one chooses a suitable
  partition of unity (typically a bounded uniform partition of unity
  \cite[Definition 11.6.1]{heil2003introduction}) subordinate to this covering,
  as indicated by \eqref{eq:_decomposition_of_frequency_space}. Renormalizing
  (cf. \cite[Proposition 3]{wiatowski2017mathematical}) this partition of
  unity yields the desired UFC filters.
  
  It turns out that for such scattering networks, energy decay is exponential, globally on $L^2(\R^d)$. Thereby, we complement the earlier result \cite[Proposition 3.3]{czaja2019analysis}, which states global exponential energy decay for scattering networks based on uniform covering frames. We have several comments concerning the similarities and differences between our following result and the result in \cite{czaja2019analysis}:
  \begin{itemize}
    \item In \cite{czaja2019analysis}, the output-generating low-pass filter is assumed to be bandlimited. Our result also applies if that is not the case.
    \item In \cite{czaja2019analysis}, the frequency support of each filter $\psi \in \Psi$ is assumed to be connected. Our result does not require connectedness of the frequency supports. However, as is pointed out in \cite[Remark 2.2]{czaja2019analysis}, the motivation to require connectedness of the frequency supports is to preclude certain pathological behavior such as $\supp(\widehat{\psi})$ having two connected components, where one component is near the origin and the other is far from the origin. In our setting, we implicitly preclude such pathological behavior by means of Assumption \ref{ass: high-pass nature and analyticity of filters}, which imposes frequency localization of the filters.
    \item If $D_\Psi<\infty$, then the frequency support of each filter $\psi \in \Psi$ is contained in a closed ball of radius $D_\Psi$. In particular, $\Psi$ satisfies the uniform covering property introduced in \cite{czaja2019analysis}, i.e., for any $R>0$, there exists $N \in \N$ such that for each $\psi \in \Psi$, the set $\supp(\widehat{\psi})$ can be covered by $N$ cubes of side length $2R$. In turn, if $\Psi$ satisfies the uniform covering property, and if the frequency supports of the filters are connected, then $D_\Psi<\infty$.
    \item The findings in \cite[Proposition 3.3]{czaja2019analysis} guarantee exponential decay of the energy remainder: For all $f \in L^2(\R^d)$, and all $N \in \N$,
    \[W_N(f)\leq \alpha^{N}\cdot \left(\norm{f}_{2}^2-\norm{f*\chi}_2^2\right),\]
    for an \textit{unspecified} constant $\alpha \in (0,1)$ depending only on the uniform covering frame. The main advantage of our result is that we can explicitly specify the values for all occurring quantities in the upper bound on the energy remainder, including a precise value for $\alpha$.
  \end{itemize}
  
  \begin{corollary}\label{cor: global exp decay for ufc filters}
    Suppose that the filters $\Psi$ have uniform frequency concentration with parameters $\gamma$, $\rho$, and $r_1$ from Assumption \ref{ass: high-pass nature and analyticity of filters}.
   
    Then we have, for all $f \in L^2(\R^d)$, and for all $N \in \N$,
    \[W_N(f) \leq \max\set{1,\frac{2\left(\left\lfloor\frac{d}{2}\right\rfloor+1\right)}{\alpha \cdot r_1}} \cdot D_\Psi \cdot \left(\norm{f}_{2}^2-\norm{f*\chi}_2^2\right) \cdot \alpha^{N},\]
    where, as before, $\alpha:=\sqrt{1-\frac{4\gamma}{(1+\gamma)^2}\cdot (1-\rho)^2}$. Asymptotically, we even have
    \[W_N(f) \in \mathcal{O}\left(\alpha^{2N}\right).\]
  \end{corollary}
  
  \begin{proof}
    Both statements (the specific bound and the asymptotic behavior) are immediate consequences of Theorem \ref{thm: Convergence rates for W_N(f)}. To derive the specific upper bound, we choose (in the notation of the theorem)
    \[\omega:(0,\infty)\to (0,\infty), \quad t \mapsto \sqrt{t} \qquad \text{and} \qquad \widehat{\vartheta}: \R^d \to \R, \quad \xi \mapsto \left(1-\norm{\xi}_2\right)_+^{\left\lfloor\frac{d}{2}\right\rfloor+1}.\] 
    Clearly, $\omega$ is a strongly $t$-dominated weight. Further, $\vartheta$ satisfies the prerequisites of Theorem \ref{thm: W_N(f) is upper-bounded by integral} with $|\widehat{\vartheta}(r_1^{-1}\,\cdot)| \leq |\widehat{\chi}|$, i.e., $C_{\chi,\vartheta}=r_1^{-1}$. Moreover, by Bernoulli's inequality we have, for all $\xi \in \R^d$, 
    \[1-|\widehat{\vartheta}(\xi)|^2\leq 2\left(\left\lfloor\frac{d}{2}\right\rfloor+1\right) \cdot \norm{\xi}_2.\]
    Finally, note that 
    \begin{align*}
      \sum_{\psi \in \Psi} \omega^{2}(d_\psi) \cdot \norm{f*\psi}_2^2 \leq \omega^2(D_\Psi) \cdot \sum_{\psi \in \Psi} \norm{f*\psi}_2^2 = D_\Psi \cdot \left(\norm{f}_{2}^2-\norm{f*\chi}_2^2\right),
    \end{align*} 
    which concludes the first part of the corollary.
    The asymptotic rate also follows directly from Theorem \ref{thm: Convergence rates for W_N(f)}, if we take 
    \[\omega:(0,\infty)\to (0,\infty), \quad t \mapsto t \qquad \text{and} \qquad \vartheta: \R^d \to \R, \quad x \mapsto e^{-\pi \cdot \norm{x}_2^2}.\] 
  \end{proof}
  
  \begin{remark}\label{rm: strict inclusion between Fourier weighted L2 and decomp space}
    In the context of the above corollary, it holds that 
    \[\mathcal{F}L^2_\omega(\R^d) \subsetneq L^2(\R^d)= \mathcal{D}_\omega(\Psi;L^2(\R^d)).\] 
    This again shows the strength of Theorem \ref{thm: Convergence rates for W_N(f)}, and stresses the role of the filter bank specificity in energy propagation in scattering networks.
  \end{remark}
  
  \subsection{Scattering with Wavelets}\label{sec: Scattering with wavelets}
  
  In this section, we apply our results from Section \ref{sec: Slow Scattering Propagation} and Section \ref{sec: Convergence rates for scattering propagation} to filter banks that are wavelet-generated.
  
  We begin by briefly reviewing scattering with directional wavelets as introduced in \cite{mallat2012group}. Let $a>1$, and let $G<O_d(\R)$ be a finite subgroup of rotations in $\R^d$ that comprises the reflection operator $-I \in G$.
  
  \begin{definition}\label{def: (a,G)-wavelet scattering}
    We say that a pair of functions $(\psi, \phi) \in \left(L^1(\R^d)\cap L^2(\R^d)\right)^2$ is \textbf{admissible for $(a,G)$-wavelet scattering} if
  \begin{align}\label{ass: wavelet admissibility for scattering}
      |\widehat{\phi}(\xi)|^2+\sum_{j=1}^\infty \sum_{M \in G} |\widehat{\psi}(a^{-j}\cdot M^{-1}\xi)|^2=1 \quad \text{a.e. } \xi \in \R^d.
  \end{align}
  \end{definition}
  
  Note that \eqref{ass: wavelet admissibility for scattering} entails that
  \[\sum_{j\in \Z} \sum_{M \in G} |\widehat{\psi}(a^{-j}\cdot M^{-1}\xi)|^2=1 \quad \text{a.e. } \xi \in \R^d,\]
  which also implicitly requires $\psi$ to have at least one vanishing moment, i.e., $\widehat{\psi}(0)=0$.
  Fix $J \in \Z$, and set
  \begin{align*}
  \Lambda(a,G,J):=\{a^j\cdot M~|~ j \in \Z_{> -J}, M \in G\}\subseteq \GL_d(\R).
  \end{align*}
  Let, for each $\lambda=a^j\cdot M \in \Lambda(a,G,J)$, 
  \[\psi_\lambda:=D_{\lambda^T}^1\psi=a^{d j} \psi(a^j \cdot M^{-1} \, \cdot \,).\]
  Moreover, set 
  \[\phi_J:=D_{a^{-J}\cdot I}^1 \phi=a^{-d J} \phi(a^{-J}\, \cdot \,).\] 
  Finally, define
  \[\Wav(\psi,a,G,J):=\set{\psi_\lambda ~|~ \lambda \in \Lambda(a,G,J)}.\]
  By construction, the filters $\Wav(\psi,a,G,J)$ together with the corresponding low-pass filter $\chi:=\phi_J$ form a semi-discrete Parseval frame. 
  In fact, we obtain the Littlewood-Paley condition \eqref{ass: Littlewood-Paley condition} from the admissibility condition \eqref{ass: wavelet admissibility for scattering} by an index shift,
  \begin{align*}
    |\widehat{\phi}_J(\xi)|^2+ \sum_{\lambda \in \Lambda(a,G,J)} |\widehat{\psi}(\lambda^{-1}\cdot \xi)|^2 
    &= |\widehat{\phi}(a^J \cdot \xi)|^2+ \sum_{j=1-J}^\infty \sum_{M \in G} |\widehat{\psi}(a^{-j}\cdot M^{-1}\xi)|^2 \\
    &= |\widehat{\phi}(a^J \cdot \xi)|^2+ \sum_{j=1}^\infty \sum_{M \in G} |\widehat{\psi}(a^{-j}\cdot M^{-1}  (a^J \cdot \xi))|^2=1 \quad \text{a.e. } \xi \in \R^d.
  \end{align*}

  \begin{definition}
    We say that the filters $\Psi$ are \textbf{wavelet-generated} if there exist a finite subgroup $G<O_d(\R)$ with $-I \in G$,  $a>1$, $J \in \Z$, and a pair $(\psi, \phi) \in \left(L^1(\R^d)\cap L^2(\R^d)\right)^2$ that is admissible for $(a,G)$-wavelet scattering such that $\Psi=\Wav(\psi,a,G,J)$. 
  \end{definition}
  
  Our main application of the negative findings from Section \ref{sec: Slow Scattering Propagation} is that energy propagation can be arbitrarily slow in wavelet scattering networks. 
  
  \begin{corollary}\label{cor: negative findings applied to wavelet scattering}
    If the scattering filters $\Psi$ are wavelet-generated, then for every $g \in L^2(\R^d)$ there exists a universal constant $C_g>0$ with the following property:
    
    For any nonincreasing null-sequence $E=(E_N)_{N \in \N}\in \R_{>0}^\N$ there exists $f_{g,E} \in L^2(\R^d)$ such that
    \begin{myenum}
      \item[a)] $\norm{f_{g,E}}_2^2\leq C_g \cdot (1+E_1)$,
      \item[b)] $\norm{f_{g,E}-g}_2^2\leq C_g \cdot E_1$, and
      \item[c)] $W_N(f_{g,E})\geq \frac{\norm{g}_2^2}{8} \cdot E_N$ for all $N\in \N$.
    \end{myenum}
    Furthermore,
    \[Y_E:=\left\{f \in L^2(\R^d) ~\middle|~ W_N(f) \in \mathcal{O}(E_N)\right\}\]
    is a countable union of nowhere dense sets in $L^2(\R^d)$. In particular, $L^2(\R^d) \setminus Y_E$ is dense in $L^2(\R^d)$.
  \end{corollary}
  
  \begin{proof}
    The proof is a straightforward application of Corollary \ref{cor: slow propagation in case of a single expansive matrix}. Suppose that $\Psi=\Wav(\psi,a,G,J)$. Setting $A:=a \cdot I_d$ gives $\sigma_{\min}(A)=a>1$. Moreover, for all $m \in \N$, applying $D_{A^{-m}}^1$ to the filters $\Psi$ results in an index shift, which leads to the required inclusion relation
    \begin{align*}
      \Psi_m:=D_{A^{-m}}^1\Psi&=\set{D_{A^{-m}}^1 \psi_\lambda ~\middle|~ \lambda \in \Lambda(a,G,J)}\\
      &=\set{D_{\lambda^T\cdot A^{-m}}^1 \psi ~\middle|~ \lambda \in \Lambda(a,G,J)}\\
      &=\set{\psi_\lambda ~\middle|~ \lambda \in \Lambda(a,G,J+m)}\\
      &\subseteq \set{\psi_\lambda ~\middle|~ \lambda \in \Lambda(a,G,J+m+1)} =\Psi_{m+1}.
    \end{align*}
  \end{proof}
  
  The fact that for any nonincreasing null-sequence $E \in \R_{>0}^\N$ the set $L^2(\R^d)\setminus Y_E$, which contains those signals whose associated energy propagates \textit{not} at the order of $E$ through the scattering network, is dense in $L^2(\R^d)$ indicates that wavelet scattering energy propagation is not a stable property of signals. This is further emphasized by the contrast that fast energy propagation provably also holds for dense signal classes of $L^2(\R^d)$ if the generating wavelet is bandlimited.

  \begin{corollary}\label{cor: Wavelet decay rates}
    Suppose that the scattering filters are wavelet-generated according to $\Psi=\Wav(\psi,a,G,J)$. Moreover, assume that there exist $M_{\psi} \in O_d(\R)$, $\rho \in [0,1)$, and $\kappa \in \N_{\geq 2}$ such that
    \[\supp(\widehat{\psi})\subseteq C_{M_{\psi}}^{\rho} \cap \s_{a^{-1},a^{\kappa-1}}.\] 
    Let \[\alpha:=\sqrt{1-\frac{4a^\kappa}{(1+a^\kappa)^2}\cdot (1-\rho)^2} \: \in (0,1).\] 
    If $\omega:(0,\infty)\to (0,\infty)$ is a weakly $t^2$-dominated weight, then we have, for all $f \in \mathcal{D}_\omega(\Psi;L^2(\R^d))$,
    \[W_N(f)\in \mathcal{O}(\omega^{-2}(\alpha^{-N})).\]
    Specifically, we can guarantee the following convergence rates:
    \begin{enumerate}[a)]
      \item If $f \in H^s(\R^d)$ for some $s>0$, then 
      \[W_N(f) \in \mathcal{O}(\alpha^{2\min\set{s,1}\cdot N}).\] 
      \item If $f \in H^s_{\log}(\R^d)$ for some $s>0$, then 
      \[W_N(f)\in \mathcal{O}(N^{-2s}).\] 
    \end{enumerate}
  \end{corollary}
  
  \begin{proof}
    We only need to show that Assumption \ref{ass: high-pass nature and analyticity of filters} is satisfied in this setting. To this end, we note that we have, for all $\lambda=a^j\cdot M \in \Lambda(a,G,J)$, 
    \begin{align*}
      \supp(\widehat{\psi_\lambda})\subseteq \lambda \cdot \supp(\widehat{\psi}) \subseteq \lambda \cdot \left(C_{M_{\psi}}^{\rho} \cap \s_{a^{-1},a^{\kappa-1}}\right) = C_{M \cdot M_{\psi}}^{\rho} \cap \s_{a^{j-1},a^{\kappa+j-1}}.
    \end{align*}
    Thus, we may apply Theorem \ref{thm: Convergence rates for W_N(f)} and Corollary \ref{cor: Decay rates for (Log)Sobolev functions}, where $\gamma=a^{-\kappa}$.
  \end{proof}
  
  \begin{remark}
    Particularizing part a) of Corollary \ref{cor: Wavelet decay rates} for an overlap of $\kappa=2$ and $\rho=0$ in dimension $d=1$ improves the currently best known upper bound on the asymptotic convergence rate for $f \in H^s(\R)$ from 
    \[W_N(f)\in \mathcal{O}\left(\left(\frac{a^2-1}{a^2+1}\right)^{\min\{2s,1\}\cdot N}\right)\] 
    given by \cite[Theorem 3.1]{wiatowski2017topology} to 
    \[W_N(f)\in \mathcal{O}\left(\left(\frac{a^2-1}{a^2+1}\right)^{2\min\{s,1\}\cdot N}\right).\]
  \end{remark}
  
  \begin{remark}
    The semantic content of a generic signal occurring in classification tasks is typically stable under the action of small diffeomorphisms that deform signals \cite{mallat2012group,mallat2016understanding}. The windowed scattering transform was introduced as a model for feature extraction, building translation invariant representations of signals in $L^2(\R^d)$ that take this form of stability into account. 
    
    One approach \cite{grohs2016deep,wiatowski2017mathematical,koller2018deformation} to prove stability of the scattering transform under the action of small diffeomorphisms relies on the non-expansiveness of $\s[\mathfrak{F}]$. In this case, the stability ultimately results from the deformation sensitivity of the class of signals being considered, and thus naturally only applies to strict subclasses of $L^2(\R^d)$.
  
    A different approach \cite{mallat2012group,nicola2023stability} explicitly takes into account the network architecture of scattering networks. Here, stability of the wavelet scattering transform under the action of small diffeomorphisms can be guaranteed for signals that have a finite mixed $(\ell^1,L^2(\R^d))$ scattering norm, i.e., for signals $f\in L^2(\R^d)$ that satisfy
    \begin{align}\label{eq: mixed l1L2 scattering norm}
      \norm{U[\mathcal{P}_\Psi](f)}_{\ell^1(\mathcal{P}_\Psi;L^2(\R^d))}=\sum_{N=0}^\infty \norm{U[\Psi^N]f}_{\ell^2(\Psi^N;L^2(\R^d))} = \sum_{N=0}^\infty \left(W_N[\Psi](f)\right)^{\frac{1}{2}}<\infty.
    \end{align}
    In general, it is not easy to decide whether a signal satisfies this condition. 
    In light of Corollary \ref{cor: negative findings applied to wavelet scattering}, there exist dense subsets of $L^2(\R^d)$ for which the mixed $(\ell^1,L^2(\R^d))$ scattering norm \eqref{eq: mixed l1L2 scattering norm} does not converge. On the other hand, if the generating wavelet of the filter bank is bandlimited, then Corollary \ref{cor: Wavelet decay rates} implies that the mixed $(\ell^1,L^2(\R^d))$ scattering norm is finite at least for all signals belonging to $H^s_{\log}(\R^d)$ for any $s>1$. This generalizes \cite[Proposition 2.4]{nicola2023stability} to arbitrary dimension $d \in \N$.
  \end{remark}
  
  \section{Concluding remarks}\label{sec: Concluding remarks}

  Our paper focuses on the study of scattering transforms for the \textit{continuous-domain setting}, i.e., transforms acting on the infinite-dimensional Hilbert space $L^2(\R^d)$. This raises the question of how relevant the results of our paper are for practical applications, necessarily dealing with finite implementations acting on finite signals or images. A complementary treatment of the discrete/finite setting is given in~\cite{getter2025digital}.
  
  At the outset of this discussion, we wish to point out that the use of continuous domains for the development and analysis of signal and image processing methods has a long-established tradition, and underlies many existing algorithms in the field. As an example of a rather large direction of research founded on such ideas, let us mention variational signal and image processing, and name \cite{rudin1992nonlinear} as a particularly influential representative of a huge body of literature. Another case in point, which is somewhat closer to the origins of the scattering transform, is wavelet signal and image processing, in particular in the context of data compression \cite{deVore1992image,donoho1998data}.

  We see at least two ways in which the continuous-domain setting can prove useful in understanding the finite signal setting. The first approach uses the continuous domain as a laboratory to develop ideas and intuitions, often based on mathematical theorems that exploit the rich analytical theory that is available for the continuous domain. The transfer to the finite, discrete-signal domain is then done mostly by analogy, rather than in the form of rigorous, quantitative results. The second, more rigorous approach tries to translate theorems about objects in the continuous-domain setting to theorems about suitable discretizations of these objects.
  
  As an illustration of the first approach, we cite the invention of the scattering transform itself, which was initially conceived for the continuous domain. It is beneficial to recall the rationale for the introduction of the scattering transform as a designed (rather than learned) convolutional neural network feature extractor. The analytical properties of the continuous-domain scattering transform serve as a template for desirable properties one wishes to exploit in applications of its discrete-domain counterpart to
  finite signals. However, most of the relevant properties of the scattering features, such as robustness with respect to deformations, would have very cumbersome formulations in a strictly discrete, finite-domain setting. Nonetheless, one expects that the analytical properties derived in the continuous-domain setting are relevant for the discrete setting as well,
  and this expectation has been corroborated in case studies for a variety of applications~\cite{anden2015joint,hirn2017quantum,bruna2013invariant}.
  
  We claim that some of the results and arguments in our paper have similarly useful interpretations in the discrete-domain setting as well. More precisely, we expect the intuitions that led to the construction of our adversarial signal class exhibiting arbitrarily slow energy decay to be of some use in the construction of similar signals in finite domain. The fundamental ideas behind our construction are that higher frequency components take longer to propagate, and that the scattering transform behaves almost additively on sums of signal components that are separated in frequency. It is conceivable (although beyond the scope of this paper) to use these intuitions to identify general signal classes with somewhat slow energy decay behavior for a finite, fixed number of layers.

  However, the finite domain limits the extent to which both relevant techniques (dilating signal components to shift them to higher frequencies, and component separation) can be employed. In fact, we do not expect the negative results of our paper to carry over verbatim to the finite-domain setting. To mention a related (if not directly applicable) recent result, \cite[Proposition 3.3]{zou2020graph} guarantees exponential energy decay for scattering transforms over finite graphs, globally for all input signals. In some sense, our Corollary \ref{cor: global exp decay for ufc filters} seems more directly relevant to the finite-domain setting: Recall that this result concludes exponential decay, as soon as the filterbank underlying the scattering transform has uniformly limited bandwidth. One can argue that over finite domains, this limited bandwidth condition is always fulfilled, and therefore the principle behind the proof of Corollary \ref{cor: global exp decay for ufc filters} (if not the proof itself) leads one to expect exponential decay; cf.~\cite{getter2025digital}.
  
  We next sketch the second, more rigorous approach of making the improved understanding of the scattering transform provided by our results useful in the discrete context. The underlying idea is that discrete signals and images are understood as discretizations of continuous-domain objects, and one expects to employ well-established analytical tools to predict the result of signal processing algorithms acting on the discretization of a given continuous-domain signal, based on the analytical properties of the latter. The main benefit is that one expects to control the impact of data resolution; e.g., to explain how the processing of different discretized versions of the same continuous-domain signal is affected by their resolution. Examples of such approaches in the context of wavelet
  data compression (with different rationales) can be found in \cite{deVore1992image,donoho1998data}.

  Transferring this reasoning to the question of energy decay for scattering transforms, one can speculate under what conditions a continuous-domain signal, for which the continuous-domain theory predicts a certain decay behavior in a given scattering transform, exhibits a comparable decay after discretizing and processing by a corresponding discrete-domain scattering transform, in a way that is \textit{independent of the resolution}. It is important to realize that this last feature is not provided by the finite-domain decay result cited above, namely \cite[Proposition 3.3]{zou2020graph}. This result formulates global exponential decay, but the basis of the exponential decay converges to one, as the dimension of the signal space (which essentially corresponds to the resolution) becomes large. In other words, this decay behavior depends on the resolution.

  From this point of view, our counterexamples serve as a caveat regarding the use of finite wavelet scattering transforms with a resolution-independent fixed depth. Note, however, that one possible way of justifying this choice consists in making the assumption that all practically relevant discrete signals are discretizations of continuous-time \textit{signals from suitable function spaces}, such as the generalized Sobolev spaces that appear in our positive results. This assumption can be understood as an implicit bias imposed on the relevant signal classes.  We point out that in the already mentioned sources on wavelet signal compression, the continuous-domain Besov spaces play a comparable role.
  Understanding whether such a bias exists, and how it is incurred by the choice of filter bank, has been one of the initial impulses of this work.

  \section*{Data availability}

  No data was used for the research described in the article.
  
  \section*{Acknowledgement}
  
  The authors acknowledge funding by the Deutsche Forschungsgemeinschaft (DFG, German Research Foundation) - Project number 442047500 through the Collaborative Research Center "Sparsity and Singular Structures" (SFB 1481).

\printbibliography
\vspace*{2em}

\end{document}